\documentclass{amsart}

\usepackage[dvipsnames]{xcolor}
\usepackage[top=3cm, bottom=2cm, left=3cm, right=3cm]{geometry}

\usepackage[utf8]{inputenc}
\usepackage[T1]{fontenc}
\usepackage{textcomp}
\usepackage{amsmath, amssymb,mathrsfs,amsthm}
\usepackage{mathtools}
\usepackage{bbm}
\usepackage{multirow}

\usepackage{tikz}
\usetikzlibrary{arrows}

\mathtoolsset{showonlyrefs}

\usepackage[linktocpage]{hyperref}
\hypersetup{	
  colorlinks=true,
  breaklinks=true,
  urlcolor= red,
  linkcolor= red,
  citecolor=ForestGreen
}

\usepackage{comment}

\numberwithin{equation}{section}

\renewcommand{\d}{\,\mathrm{d}}

\newcommand{\un}{\ensuremath{\mathbbm{1}}}

\newcommand\N{\ensuremath{\mathbb{N}}}
\newcommand\R{\ensuremath{\mathbb{R}}}
\newcommand\Z{\ensuremath{\mathbb{Z}}}
\newcommand\T{\ensuremath{\mathbb{T}}}

\newcommand\C{\ensuremath{\mathbb{C}}}

\newcommand\qc{\ensuremath{\mathcal{Q}}}
\newcommand\lc{\ensuremath{\mathcal{L}}}

\newcommand\fc{\ensuremath{\mathcal{F}}}

\newcommand\dst{\displaystyle}

\newcommand\lb{\ensuremath{{\displaystyle(}}}
\newcommand\rb{\ensuremath{{\displaystyle)}}}

\newtheorem{theorem}{Theorem}[section]
\newtheorem{lemma}[theorem]{Lemma}
\newtheorem{proposition}[theorem]{Proposition}

\newtheorem{corollary}[theorem]{Corollary}
\numberwithin{equation}{section}
\newtheorem{taggedtheoremx}{Assumption}
\newenvironment{taggedtheorem}[1]
 {\renewcommand\thetaggedtheoremx{#1}\taggedtheoremx}
 {\endtaggedtheoremx}
 
 \theoremstyle{definition}
 \newtheorem{remark}[theorem]{Remark}
 \newtheorem{definition}[theorem]{Definition}
 \newtheorem{notation}[theorem]{Notation}

\newcommand{\norm}[1]{{\left\|{#1}\right\|}}
\newcommand{\ent}[1]{{\left[{#1}\right]}}
\newcommand{\abs}[1]{{\left|{#1}\right|}}
\newcommand{\scal}[1]{{\left\langle{#1}\right\rangle}}

\begin{document}
\title{Null-controllability of the Generalized Baouendi-Grushin heat like equations}

\author[P. Jaming]{Philippe Jaming}
\address{Univ. Bordeaux, CNRS, Bordeaux INP, IMB, UMR 5251, F-33400 Talence, France}
\email{philippe.jaming@math.u-bordeaux.fr, yunlei.wang@math.u-bordeaux.fr}

\author[Y. Wang]{Yunlei Wang}

\begin{abstract}
	In this article, we prove null-controllability results for the heat equation associated to
	fractional Baouendi-Grushin operators 
$$
\partial_t u+\bigl(-\Delta_x-V(x)\Delta_y\bigr)^s u=\un_\Omega h
$$	
where $V$ is a potential that satisfies some power growth conditions and the set $\Omega$
is thick in some sense. This extends previously known results for potentials $V(x)=|x|^{2k}$.

To do so, we study Zhu-Zhuge's spectral inequality for Schrödinger operators with power growth potentials, and give a precised quantitative form of it.
\end{abstract}

\maketitle

\tableofcontents

\section{Introduction}

The aim of this paper is to prove null-controllability and observability  from equidistributed subsets of $\R^d$
for heat equations associated to Baoeundi-Grushin type
operators 
$$
\lc_V(x,y)=-\Delta_x-V(x)\Delta_y,\qquad x\in\R^n\mbox{ and }y\in\R^m\mbox{ or }\T^m
$$
where the potential $V$ has a power growth $c_1|x|^{\beta_1}\leq V(x)\leq c_2|x|^{\beta_2}$ and an upper bound on the 
gradiant (for precise assumptions on $V$, {\it see} Assumptions \ref{A1}-\ref{A2} below). 
Let us now make this more precise.

\subsection{Null-controllability, observability and spectral inequalities}

First recall that null-control\-la\-bility is defined as follows:

\begin{definition} [Null-controllability] Let $P$ be a closed operator on $L^2(\R^d)$ which is the infinitesimal generator of a strongly continuous semigroup $(e^{-tP})_{t \geq 0}$ on $L^2(\R^d)$, $T>0$ and $\Omega$ be a measurable subset of $\mathbb{R}^d$. 
The equation 
\begin{equation}\label{syst_general}
\left\lbrace \begin{array}{ll}
(\partial_t + P)u(t,x)=h(t,x)\un_{\Omega}(x), \quad &  x \in \mathbb{R}^d,\ t>0, \\
u|_{t=0}=u_0 \in L^2(\R^d),                                       &  
\end{array}\right.
\end{equation}
is said to be {\em null-controllable from the set $\Omega$ in time} $T>0$ if, for any initial datum $u_0 \in L^{2}(\mathbb{R}^d)$, there exists $h \in L^2((0,T)\times\mathbb{R}^d)$, supported in $(0,T)\times\Omega$, such that the mild (or semigroup) solution of \eqref{syst_general} satisfies $u|_{t=T}=0$.
\end{definition}

\medskip

By the Hilbert Uniqueness Method, {\it see} \cite[Theorem~2.44]{coron_book} or \cite{JLL_book,TW_book}, the null controllability of the equation \eqref{syst_general} is equivalent to the (final state) observability of the adjoint system 
\begin{equation} \label{adj_general}
\left\lbrace \begin{array}{ll}
(\partial_t + P^*)v(t,x)=0, \quad & x \in \mathbb{R}^d,\ t>0, \\
v|_{t=0}=v_0 \in L^2(\R^n),
\end{array}\right.
\end{equation}
where $P^*$ stands for the $L^2(\R^d)$-adjoint of $P$.
This notion of observability is defined as follows:

\medskip

\begin{definition} [Observability] Let $T>0$ and $\Omega$ be a measurable subset of $\mathbb{R}^d$. 
Equation \eqref{adj_general} is said to be {\em observable from the set $\Omega$ in time} $T>0$ if there exists a positive constant $C_T>0$ such that,
for any initial datum $v_0 \in L^{2}(\mathbb{R}^d)$, the mild (or semigroup) solution of \eqref{adj_general} satisfies
\begin{equation}\label{eq:observability}
\int\limits_{\mathbb{R}^d} |v(T,x)|^{2}\,\mathrm{d}x 
 \leq C_T \int\limits_{0}^{T} \Big(\int\limits_{\Omega} |v(t,x)|^{2} \,\mathrm{d}x\Big) \,\mathrm{d}t.
\end{equation}
\end{definition}

\medskip

Null-controllability/observability problems of heat equations associated to elliptic operators on domains
is a an old and vast subject ({\it see} \cite{coron_book,JLL_book,TW_book} and references therein). More recently,
following the seminal work by K. Beauchard, P. Cannarsa and R. Guglielmi \cite{cannarsa2013null},
research about null-controllability properties of heat equations associated to degenerate
hypo-elliptic operators of the Baouendi-Grushin type have attracted a lot of attention ({\it see e.g.}
\cite{FMbook,koenig2017non,duprez2020control,darde2022null,beauchard20152d,beauchard2020minimal,allonsius2021analysis,lissy2022non}).

In another direction, there has been a lot of recent activity on
null-controllability/observability problems for
operators $P$ defined on the whole space $\R^d$ and sets $\omega$ satisfying various thickness conditions.
The first result in this direction concerns the heat equation on $\R^d$ ($P=\Delta$), and observability
from so-called thick sets. To define this notion, let us introduce the following notation:
we denote by $\mathcal{B}_r(z)$ the Euclidean ball with radius $r$ centered on $z$ and 
$\mathcal{B}_r=\mathcal{B}_r(0)$. We write $\mathcal{Q}_r=[-r/2,r/2]^d$ and $\mathcal{Q}_r(z)=z+\mathcal{Q}_r$ denotes the cube with side length $r$ centered on $z$.
The Lebesgue measure of a set $E\subset\R^d$ will be denoted by $|E|$. 

\begin{definition}
Let $\gamma \in (0,1)$ and $\ell>0$. A set $\omega\subset\R^d$ is said to be $\gamma$-thick (at scale $\ell>0$) if, for every $z\in\R^d$,
$$
|\omega\cap(z+\mathcal{Q}_\ell)|\geq\gamma|\mathcal{Q}_\ell|.
$$
We will say that $\omega$ is thick, if it is $\gamma$-thick at scale $\ell$ for some $\gamma\in(0,1)$ and $\ell>0$.
\end{definition}

It is easy to check that the notion of thickness does not depend on the norm used to define it, in particular
cubes may be replaced by balls. In other words, $\omega$ is thick 
if and only if there is a $\gamma'\in(0,1)$ and an $\ell'>0$
such that, for every $z\in\R^d$, $|\omega\cap(z+\mathcal{B}_{\ell'})|\geq\gamma'|\mathcal{B}_{\ell'}|$.

It was then proven in \cite{egidi2018sharp} and independently in  \cite{wang2019observable}
that the heat equation on $\R^d$ is controllable/observable from a set $\omega$ if and only if $\omega$
is thick:

\begin{theorem}[Egidi \& Veseli\'c \cite{egidi2018sharp}, Wang, Wang, Zhang \& Zhang \cite{wang2019observable}]
\label{th:WWZZEV}
Let $\omega\subset\R^d$ and $T>0$. Then the following are equivalent:
\begin{enumerate}
\renewcommand{\theenumi}{\roman{enumi}}
\item $\Omega$ is thick;

\item {\it Observability:} there exists a constant $C>0$ such that, for every $f\in L^2(\R^d)$,
$$
\norm{e^{-T\Delta}f}_{L^2(\R^d)}^2\leq C\int_0^T\norm{e^{-t\Delta}f}_{L^2(\Omega)}^2\,\mathrm{d}t.;
$$

\item {\it Controllability:} For every $u_0\in L^2(\R^d)$ there exists $h\in L^2\bigl((0,T)\times\Omega\bigr)$
such that the solution of the heat equation $\partial_t u-\Delta_xu=h(t,x)\un_\Omega(x)$
with initial condition $u(0,x)=u_0(x)$ satisfies $u(T,x)=0$.
\end{enumerate}
\end{theorem}

Soon after, the case of heat equations associated to
some Schrödinger operators
\[
P=H_V:=-\Delta_x +V(x),\quad  x \in \R^{d},
\] 
started to be investigated. For $V=|x|^2$, $P$ is the harmonic oscillator and observability inequalities from different kinds of $\omega$ were established, {\it see} \cite{beauchard2018null,beauchard2021spectral,egidi2021abstract,martin2022spectral,dicke2023uncertainty}. 
For $V=|x|^{2k}$, null controllability and observability inequalities were established for $P=H_V$ in \cite{alphonse2023quantitative,alphonse2020null,martin2022spectral,miller,martin}.
Very recently, in \cite{alphonse2023quantitative}, P. Alphonse and A. Seelman extended
some of those results  to Baouendi-Grushin operators $P=\lc_{|x|^{2k}}$. Our aim here is to pursue
this line of research by extending the class of potentials $V$.

All those results require various restrictions on $\omega$.
Apart from thick sets that we already defined, we now need to introduce two further classes.
The first one was introduced in control theory in \cite{rojas2013scale} and consists
of a particular class of thick sets and are those one would naturally consider if
one wants to construct such sets:

\begin{definition}
Let $\gamma \in (0,1)$ and $0\leq\sigma <1$. 
\begin{enumerate}
\renewcommand{\theenumi}{\roman{enumi}}

\item A set $\omega\subset\R^d$ is said to be {\em $\gamma$-equidistributed} or simply {\em equidistributed }
if there exists a set of points $\{z_k\}_{k\in \Z^{d}}\subset\R^d$ such that,
for every $k$
$$
\omega \cap (k+\mathcal{Q}_1)\supset \mathcal{B}_{\gamma}(z_k).
$$

Since \cite{rojas2013scale}, a larger class has appeared:

\item $\omega$ is {\em $(\gamma,\sigma)$-distributed} if there exists a set of points $\{z_k\}_{k\in \Z^{d}}\subset\R^d$ such that,
for every $k\in\Z^d$,
	\begin{equation}\label{1.3a}
	z_k\in k+\mathcal{Q}_1\quad\mbox{and}\quad
\omega \cap (k+\mathcal{Q}_1)\supset \mathcal{B}_{\gamma^{1+|k|^{\sigma }}}(z_k).
	\end{equation}
\end{enumerate}
Note that a $(\gamma,0)$-distributed set is $\sqrt{\gamma}$-equidistributed.
\end{definition}

Next, thanks to the famous Lebeau-Robbiano method introduced in \cite{lebeau1995controle} ({\it see also} \cite{tenenbaum2011null,beauchard2018null,nakic2020sharp,gallaun2020sufficient}),
the mentioned results are based on some form of Spectral Inequality. To be more precise, recall
that a spectral inequality for a nonnegative selfadjoint operator $P$ in $L^2(\R^{d})$ takes the form
\begin{equation}\label{1a}
\|\phi\|_{L^2(\R^{d})}^2\le \kappa_0 e^{\kappa_1 \lambda^{ \zeta}}\|\phi\|^2_{L^2(\omega)},\quad \forall \phi \in \mathcal{E}_P(\lambda),\,\lambda\ge 0,
\end{equation}
where $\omega$ is a measurable subset of $\R^{d}$, $\mathcal{E}_P(\lambda)$
are the spectral sets
associated to $P$, and $\kappa_0$, $\kappa_1$, 
$\zeta$ are constants. Such an inequality is a quantitative version of a unique continuation property (i.e., $f=0$ on $\omega$ implies $f=0$ on $\R^{d}$). 

In the case of $P=-\Delta$ being the Laplacian on $\mathbb{R}^d$, this inequality is a reformulation of
Kovrijkine's sharp version of the Logvinenko-Sereda Uncertainty Principle \cite{kovrijkine2001some}
and is valid if (and only if) $\omega$ is $\gamma$-thick for some $\gamma>0$.
The strategy of proof of Kovrijkine is very powerfull and the authors of several of the results mentionned so far have
managed to implement this strategy in different settings. 

Alternative strategies are based on Carleman estimates and allow for potentials $V$ that are less regular at the price
of having slightly more regular sets $\omega$. In this direction, Dicke, Seelmann and Veseli\'{c} \cite{dicke2022spectral} recently considered the Schrödinger operator with power growth potentials and a set $\omega$ that is $(\gamma,\sigma)$-distributed. Precisely speaking, they establish the spectral inequality \eqref{1a} 
when the potential $V\in W_{\mathrm{loc}}^{1,\infty}(\R^{d})$ has suitable power growth.

Shortly after, Zhu and Zhuge \cite{zhu2023spectral} improved the exponent of $\lambda$ in \eqref{1a}
for a slightly larger class of potentials (thus positively answering to the questions asked in \cite{dicke2022spectral}), namely those satisfying the following assumption:

\begin{taggedtheorem}{A}\label{A0} 
$V\in L^1_{\mathrm{loc}}(\R^d)$ is real valued and there are constants $c_1,c_2>0$ and $0<\beta_1\leq\beta_2$
such that
\begin{enumerate}
\renewcommand{\theenumi}{\roman{enumi}}
	\item for every $x\in \R^{d}$,
		\begin{equation}
			c_1 (|x|-1)_+^{\beta_1}\le V(x).\label{1.0}
		\end{equation}
		where $(a)_+ :=\max \left\{a,0\right\} $;
	\item we can write $V=V_1+V_2$ with
		\begin{equation}
			|V_1(x)|+|\nabla V_1(x)|+|V_2(x)|^{\frac{4}{3}}\le c_2(|x|+1)^{\beta_2}
		\end{equation}
		for every $x\in\R^d$.
\end{enumerate}
\end{taggedtheorem}

The main result of Zhu and Zhuge \cite{zhu2023spectral} is then that, when $P=H_V:=-\Delta+V$ with $V$
satisfying assumption \ref{A0} and $\omega$ a $(\gamma,\sigma)$-distributed set,
\begin{equation}
\label{eq:speczz}
\|\phi\|_{L^2(\R^{d})}^2\le \kappa_0 \left(\frac{1}{\gamma}\right)^{\kappa_1 \lambda^{ \zeta}}\|\phi\|^2_{L^2(\omega)},\quad \forall \phi \in \mathcal{E}_P(\lambda),\,\lambda\ge 0,
\end{equation}
where $\zeta=\dfrac{2\sigma+\beta_2}{2\beta_1}$ (thus improving the exponent
$\zeta=\dfrac{3\sigma+2\beta_2}{3\beta_1}$ in \cite{dicke2022spectral}).
Note for future use that \cite{zhu2023spectral,dicke2022spectral}
do not provide estimates of the constants $\kappa_0,\kappa_1$ in \eqref{1a} in terms of $c_1,c_2$ that we will
need here. We will thus revisit their proofs in Theorem \ref{thm1.1} below in order to obtain those estimates.


Also, in view of the results by L. Miller \cite{miller} and J. Martin \cite{martin}, one should expect to
be able to improve those results when $\beta_1>2$ by taking much smaller sets $\omega$ (in the case
of the Baouendi-Grushin operator, $\omega$ can be any open set). We do not pursue in this direction
here.

\subsection{Generalized Baouendi-Grushin operator}

Let us now move to the focus of this article. We will here consider observability/controllability properties of
evolution equations
\begin{equation}
	\partial_t u + Pu=0\label{1.5b}
\end{equation}
where $P$ is an operator of Baouendi-Grushin type associated to a real-valued potential $V$,
$$
P=\begin{cases}\lc_V:=-\Delta_x-V(x)\Delta_y&x\in\R^n,\ y\in \R^m \\
\lc_{V,p}:=-\Delta_x-V(x)\Delta_y&x\in\R^n,\ y\in \T^m
\end{cases}
$$
where $\T=\R /2 \pi\Z$ (by abuse of notation, we will write $\lc_V$ for both operators). 

The potential $V$ will satisfy some assumptions which are a bit stronger then Assumption \ref{A0}
that covers both the Grushin operator
\begin{equation*}
	\Delta_G=\lc_{|x|^2}:=-\Delta_x-|x|^{2}\Delta_y
\end{equation*}
and the Baouendi-Grushin operator
\begin{equation*}
	\Delta_k=\lc_{|x|^{2k}}:=-\Delta_x-|x|^{2k}\Delta_y,\qquad k\mbox{ a positive integer}.
\end{equation*}

The first assumption we will consider is satsified in the above case $V(x)=|x|^{2k}$, $k\in\N^*$
and more generally when $V(x)=|x|^\beta$, $\beta>1$
and is defined as follows:

\begin{taggedtheorem}{A1}\label{A1}
$V\in L^{\infty}_{\mathrm{loc}}(\R^{n})$ and there exist positive $c_1,c_2,\beta_1,\beta_2>0$ such that,
for every $x\in\R^n$,
	\begin{equation}
		c_1|x|^{\beta_1}\le V(x) \,\text{ and } \,|V(x)|+|\nabla V(x)|\le c_2|x|^{\beta_2}.
	\end{equation}
\end{taggedtheorem}

Next we notice that $V(x)=|x|^\beta$ does not satisfy this assumption when $0<\beta\leq1$ and
we will replace Assumption \ref{A1} with the following weaker one:

\begin{taggedtheorem}{A2}\label{A2}
$V\in L^{\infty}_{\mathrm{loc}}(\R^{n})$ and there exist positive $c_1,c_2,\beta_1,\beta_2>0$ such that
the following two conditions are satisfied:
\begin{enumerate}
\renewcommand{\theenumi}{\roman{enumi}}
	\item for every $x\in \R^{n}$,
		\begin{equation}
			c_1 |x|^{\beta_1}\le V(x).\label{1.1}
		\end{equation}
	\item We can write $V=V_1+V_2$ such that for every $x\in\R^n$,
		\begin{equation}
			|V_1(x)|+|\nabla V_1(x)|+|V_2(x)|^{\frac{4}{3}}\le c_2(|x|+1)^{\beta_2}.
		\end{equation}
\end{enumerate}
\end{taggedtheorem}

\begin{remark}\label{r1.2}
Notice that, as $c_1>0$, $V=0$ is excluded.

Further, if $V$ satisfies Assumption \ref{A1}, then setting $V_1=V$ and $V_2=0$ we obtain that $V$
also satisfies Assumption \ref{A2}. As expected, our results will be stronger under Assumption \ref{A1}
than under Assumption \ref{A2}.

On the other hand, let us consider the standard case $V=|x|^{\beta}$ with $0<\beta\leq 1$. Take
a smooth cut-off function $\eta$ such that $\eta =1$ in $\mathcal{B}_1$ and $\eta =0$ in 
$\R^{n} \backslash \mathcal{B}_2$ and write $V_1=|x|^{\beta}(1-\eta (x))$ and $V_2(x)=|x|^{\beta}\eta (x)$,
then we see that $V$ satisfies Assumption \ref{A2}. This is the main reason for introduction of Assumption
\ref{A0} in \cite{zhu2023spectral}.

Also it should be noted that in \cite{zhu2023spectral} the original Assumption \ref{A0} is different from Assumption \ref{A2} only in \eqref{1.0}, in which the lower bound is $c_1\left( |x|-1 \right)_+^{\beta_1}$ instead of $c_1|x|^{\beta_1}$ in \eqref{1.1}. 
This will play a key role in order to obtain a lower bound of the first eigenvalue of the Schr\"odinger operator
$-\Delta+ V(x)$ with suitable dependence on the parameter $c_1$
({\it see} Subsection~\ref{subsec3.1} for details).
\end{remark}

Our main result gives conditions for null-controllability 
of the heat equation associated to $\bigl(-\Delta_x-V(x)\Delta_y\bigr)^s$ from sets of the form $\omega\times\R^m$,
$\omega$ being $\gamma$-equidistributed. 

\begin{theorem}\label{th:main}
Let $\gamma\in(0,1/2)$, $\sigma\geq0$, $s,T>0$, $c_1,c_2,\beta_1,\beta_2>0$.
Take $V\in L^{\infty}_{\mathrm{loc}}(\R^{n})$ satisfying Assumption \ref{A2}
with parameters $c_1,c_2,\beta_1,\beta_2$. Let $s_*=\dfrac{\beta_1+2}{3}$.

Let $\omega\subset\R^n$ be a $\gamma$-equidistributed set. Then the fractional Baouendi-Grushin heat equation associated to $V$
\begin{equation}
\label{egb}\tag{$E_{BG,s}$}
\left\{
\begin{array}{l}
\begin{aligned}
\partial_t u(t,x,y)+\bigl(-\Delta_x-V(x)\Delta_y\bigr)^s u(t,x,y)=h(t,x,y)&\un_{\omega\times\R^m}(x,y)\\
&\mbox{for }t>0,\,\, (x,y)\in \R^{n}\times \R^{m},\\
\end{aligned}\\
u(0,x,y)=u_0(x,y)\quad \mbox{for }(x,y)\in \R^{n}\times \R^{m}
\end{array}\right.
\end{equation}
with initial condition $u_0\in L^2(\R^{n}\times \R^{m})$
is exactly null-controllable from $\omega\times\R^m$ in any time $T>0$ if $\beta_2=\beta_1$ and $s>s_*$.

If $V$ further satisfies Assumption \ref{A1}, then the same holds with the critical
power $s_*=\dfrac{\beta_1+2}{4}$ instead of $\dfrac{\beta_1+2}{3}$.
\end{theorem}

Our second result is that, in the presence of a second potential, the set $\omega$
can be sparser. More precisely, we have the following:

\begin{proposition}\label{th:BGS}
Let $\gamma\in(0,1/2)$, $\sigma\geq0$, $T>0$, $c_1,c_2,\beta_1,\beta_2>0$ and $s>\dfrac{\beta_2+2\sigma}{2\beta_1}$.
Take $V,\tilde V\in L^{\infty}_{\mathrm{loc}}(\R^{n})$ satisfying Assumption \ref{A2}
with parameters $c_1,c_2,\beta_1,\beta_2$. Let $s_*=\dfrac{\beta_1+2}{3}$.
Let $\omega\subset\R^n$ be a $(\gamma,\sigma)$-distributed set.
Then the fractional Baouendi-Grushin-Schr\"odinger heat equation associated to $V$ and $\tilde V$
\begin{equation}
\label{egbs}\tag{$E_{BGS,s}$}
\left\{
\begin{array}{l}
\begin{aligned}
\partial_t u(t,x,y)+\bigl(-\Delta_x-V(x)\Delta_y+\tilde V(x)\bigr)^s u(t,x,y)=h(t,x&,y)\un_{\omega\times\R^m}(x,y)\\
&\mbox{for }t>0,\,\, (x,y)\in \R^{n}\times \R^{m},\\
\end{aligned}\\
u(0,x,y)=u_0(x,y)\quad \mbox{for }(x,y)\in \R^{n}\times \R^{m}
\end{array}\right.
\end{equation}
with initial condition $u(0,\cdot ,\cdot )=u_0\in L^2(\R^{n}\times \R^{m})$
is  exactly null-controllable in any time $T>0$ if $\beta_2=\beta_1$ and $s>s_*$.

If $V$ further satisfies Assumption \ref{A1}, then the same holds with the critical
power $s_*=\dfrac{\beta_1+2}{4}$ instead of $\dfrac{\beta_1+2}{3}$.
\end{proposition}

Finally, we also obtain a result on $\R^n\times\T^m$, that is, when the equation is periodic in the $y$
variable. Here the sets $\omega$ need again to be $\gamma$-equidistributed, but there is a relaxation
on the potential as the upper bound can now be of a different order than the lower bound ($\beta_2>\beta_1$):

\begin{theorem}\label{th:mainperiodic}
Let $\gamma\in(0,1/2)$, $\sigma\geq0$, $s,T>0$, $c_1,c_2,\beta_1,\beta_2>0$.
Take $V\in L^{\infty}_{\mathrm{loc}}(\R^{n})$ satisfying Assumption \ref{A2}
with parameters $c_1,c_2,\beta_1,\beta_2$. Let $s_*=\dfrac{\beta_1+2}{3}$.

Let $\omega\subset\R^n$ be a $\gamma$-equidistributed set. Then the semi-periodic fractional Baouendi-Grushin heat equation associated to $V$
\begin{equation}
\label{egbp}\tag{$E_{pBG,s}$}
\left\{
\begin{array}{l}
\begin{aligned}
\partial_t u(t,x,y)+\bigl(-\Delta_x-V(x)\Delta_y\bigr)^s u(t,x,y)=h(t,x,y)&\un_{\omega\times\R^m}(x,y)\\
&\mbox{for }t>0,\,\, (x,y)\in \R^{n}\times \T^{m},\\
\end{aligned}\\
u(0,x,y)=u_0(x,y)\quad \mbox{for }(x,y)\in \R^{n}\times \T^{m}
\end{array}\right.
\end{equation}
with initial condition $u(0,\cdot ,\cdot )=u_0\in L^2(\R^{n}\times \T^{m})$
is exactly null-controllable in any time $T>0$ if $s>s_*$.

If $V$ further satisfies Assumption \ref{A1}, then the same holds with the critical
power $s_*=\dfrac{\beta_1+2}{4}$ instead of $\dfrac{\beta_1+2}{3}$.
\end{theorem}

\begin{remark}
For the sake of this last theorem, the main difference between the Baouendi-Grushin
operator on $\R^n\times\R^m$ and $\R^n\times\T^m$ is that $L^2(\T^m)$ has an orthonormal
basis of eigenvectors of the Laplace operator
on $\T^m$. The same result would be true e.g. if $\T^m$
is replaced by any compact Riemannian manifold $S$ and $\Delta_y$ the
Laplace-Beltrami operator on $S$.

Our proof does not extend to controllability from $(\gamma,\sigma)$-distributed sets
mainly because $0$ is an eigenvalue. This difficulty no longer arises when adding a second potential
$\tilde V$ satisfying the same assumption as $V$ and considering
$$
\begin{aligned}
		\partial_t u(t,x,y)+\bigl(-\Delta_x-V(x)\Delta_y+\tilde V(x)\bigr)^s u(t,x,y)
		=h(t,x,y)&\un_{\omega\times\T^m}(x,y)
		\\
		 &t>0,\,\, (x,y)\in \R^{n}\times \T^{m},\nonumber
\end{aligned}
$$
instead of \eqref{egbp}.
\end{remark}

The results have been recasted in the synthetic Table \ref{table.1}, Section \ref{secfinalproof}
where we will show the corresponding obsevability inequalities.

\begin{remark}\label{r1h}
In the standard case of $V(x)=|x|^\beta$, the
null-controllability from sets of the form $\tilde\omega=\omega\times \R^m$ resp. $\tilde\omega=\omega\times \T^m$
of the evolution equation
\begin{equation}
\label{std.1}\tag{$E_{\beta,s}$}
\begin{aligned}
		\partial_t u(t,x,y)+\bigl(-\Delta_x-|x|^\beta\Delta_y\bigr)^s u(t,x,y)=h(t,x,y)\un_{\omega\times\T^m}(x,y)
		\\
		 t>0,\ x\in \R^n,\ y\in \R^{n}\mbox{ resp. }\T^{m},\nonumber
\end{aligned}
\end{equation}
with initial condition $u(0,\cdot ,\cdot )=u_0\in L^2(\R^{n}\times \R^{m})$ (resp. $ L^2(\R^{n}\times \T^{m})$)
is summarized in the following picture:

\begin{figure}[htpb]
	\centering
	\begin{tikzpicture}
		\path[fill=gray!60] (0, 1) node[left] {\small $\displaystyle\frac{2}{3}$} -- (2.0,2) --(1.9,5) -- (0,5) -- (0, 4 /3);
		\path[fill=gray!60] (1.9,3 /2)  -- (1.9,5)  -- (5,5) -- (5, 9 /4) -- (1.9,3 /2);
		\draw[<->,thick] (0,5) node (yaxis) [above] {$s$}
			|- (5,0) node (xaxis) [right] {$\beta$} ;
		\draw[thick,dashed,o-o] (-0.1,1) -- (2,2);
		\path (0.3,1.2) -- (2,2) node[midway,above,sloped] {\small $s=(\beta +2) /3$};
		\draw[thick,dashed,o-] (1.81, 3 /2) -- (5,9 /4)node[midway,above,sloped] {\small $s=(\beta +2) /4$};
		\draw[dashed] (1.9,0) node[below] {\small $1$} -- (1.9,2);
		\draw (.5,3.5) node[right] {Exactly null-controllable};
		\draw (1.9, 1.3) node[right] {\small $\displaystyle (1,\frac{3}{4})$};
	\end{tikzpicture}
	\caption{The standard case}
	\label{fig:1}
\end{figure}
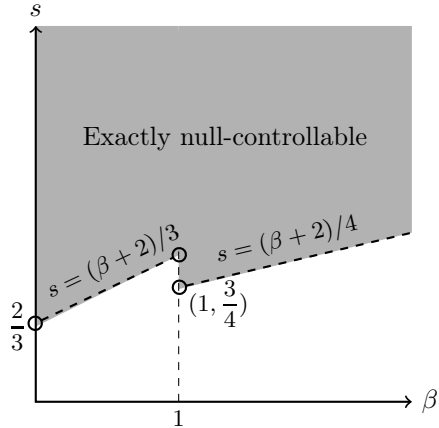

Our results thus partially recover those of Alphonse and Seelmann \cite{alphonse2023quantitative}.
There are two main losses compared to their results: one is that our control sets are a bit less general, and another is that we are not able to obtain the null-controllability of equations for the critical case $s=s_*$. The main gain is that our potentials
are more general.

Further, when $\beta=2k$ and $s=(\beta+2)/4$, Alphonse and Seelmann \cite[Theorem~2.17]{alphonse2023quantitative} show that the evolution equation is never exactly null-controllable from any control support $\omega \subset \R^{n}\times \T^{m}$ satisfying $\overline{\omega}\cap \left\{x=0\right\} =\emptyset$. This implies $s= (\beta +2) /4$ is the critical value for $\beta=2k \ge 2$. Based on our results, it is reasonable to conjecture that it is the critical value for all  
$\beta \ge 1$ under Assumption~\ref{A1} and $\beta=\beta_1=\beta_2$.

For the case $0<\beta <1$, it seems that the critical value may be a little worse. Indeed, we only obtain exactly null-controllability results up to $s=\dfrac{\beta +2}{3}>\dfrac{\beta +2}{4}$. 
One of the possible reasons may be the wilder behavior of $V=|x|^{\beta}$ around $0$.

Our results presented above are in line with articles devoted to the study of null-controllability of \eqref{std.1} in bounded domains, {\it see} \cite{cannarsa2013null,allonsius2021analysis,beauchard2020minimal,beauchard20152d,darde2022null,duprez2020control,koenig2017non}. All results in these articles give a fact that the null-controllability of \eqref{std.1} is governed by minimal time in the critical case $s=(\beta_1 +2) /4$. Note that we have a larger critical point $s= (\beta_1 +2) /3$ in Theorem~\ref{th:main} under Assumption~\ref{A2}. 
It is worthy to mention that in \cite[Theorem~1.2]{lissy2022non} the author considers the case of dimension $n=m=1,\beta=2$ and $s=1$, and then proved the equation \eqref{std.1} is never exactly null-controllable from any control support of the form  $\R\times \omega$ whenever $\overline{\omega}\neq \R$.
\end{remark}

\subsection{Strategy of proof and outline of the work}

Our strategy of proof is near to the one developped in Alphonse and Seelmann's work in \cite{alphonse2023quantitative}
and based on an earlier idea of Beauchard, Cannarsa and Guglielmi \cite{cannarsa2013null}.
To avoid technicalities, we concentrate on the non-fractional case ($s=1$).

We observe that the operator $-\Delta_x+V(x)\Delta_y$ is self-adjoint when $V$ is real valued
so that null-controllability is equivalent to observability.
The first step consists in taking a partial Fourier transform in the $y$-variable of $u$:
$$
u^\lambda(t,x)=\int_{\R^m}u(t,x,y)e^{-i\scal{y,\lambda}}\,\mbox{d}y
$$
so that we are lead to observability of the heat equation associated to Schr\"odinger operator
$$
\partial_tu^\lambda-\Delta_x u^\lambda-\lambda^{2}V(x)u^\lambda=0.
$$
We can then use results on the observability of this equation, in particular those of Zhu-Zhuge
\cite{zhu2023spectral}. 
Note that the potential $V_\lambda=\lambda^{2}V(x)$ now depends on the frequency parameter $\lambda$. 
The main difficulty is precisely to deal with this parameter. This has been done by
Alphonse and Seelmann when $V$ is a power function and is more difficult here
and requires two key differences in the proof:

\begin{enumerate}
\item We use a different kind of spectral inequality, which we call Zhu-Zhuge's inequality given in 
\cite[Theorem~1]{zhu2023spectral}. However, the original form of this theorem cannot be used directly, 
since it does not provide explicit dependence of the cost constant on the parameters in Assumptions \ref{A1}-\ref{A2}.
Our first task is thus to convert Zhu-Zhuge's proof into a more quantitative one which
leads us to an explicit form of the cost constant in the spectral inequality in terms of the parameters.
This is vital for our proofs in order to invert the partial Fourier transform in the $y$-variable.
This allows to take into account the frequency parameter $\lambda$.

\item The lowest eigenvalue of the Schrödinger operator $-\Delta+V$ is easily obtained in the case $V=c|x|^{2k},k \in \N$ by the rescaling approach. It does not work for our general potentials under Assumption \ref{A1}
or \ref{A2}. To overcome this difficulty, we do not calculate the exact value of the lowest eigenvalue, instead we just calculate a lower bound which satisfies our needs. This allows us to deal with more general potentials.
\end{enumerate}

In Theorem~\ref{th:main} the set $\omega$ is $(\gamma,0)$-distributed rather than a sparser
$(\gamma,\sigma)$-distributed set. This comes from the fact that the spectral
multiplier associated to $-\Delta_y$, namely $|\lambda|^2$ degenerates at $\lambda=0$.
This singularity is erased when adding a second potential as in Proposition \ref{th:BGS}
and explains why sparser sets are allowed there.

\smallskip

The remaining of the paper is organised as follows:

Section \ref{sec3d} is devoted to the spectral theory of Schr\"odinger operators
$H_{rV}=-\Delta+rV$ where $V\in L^\infty_{\mathrm{loc}}(\R^n)$ with a polynomial growth
$V(x)\geq c|x|^\beta$, $c,\beta>0$ and a scaling parameter $r>0$.
In the Section \ref{subsec3.1}, we provide a lower bound of the lowest eigenvalue of
$H_{rV}$ in terms of the scaling parameter while in the Section \ref{subsec2.2h} we then
provide a more quantitative version of the localization property of eigenfunctions
of $H_V$ first established in
\cite[Theorem 1.4]{dicke2022spectral} and \cite[Proposition 3]{zhu2023spectral}.

Section \ref{sec4d} is then devoted to the precised form of Zhu-Zhuge's Spectral Inequality
\eqref{eq:speczz} in which we clarify the dependence of the constant $\kappa_1$ appearing there
in terms of the parameters $c_1,c_2$ (and $\beta_1,\beta_2$).
This allows us in Section \ref{subsec.5} to obtain precised version of \eqref{eq:speczz} in which $\kappa_1$ is estimated 
in terms of
the scaling parameter $r$ of the potential {\it i.e.} when applying \eqref{eq:speczz} to $H_{rV}$
instead of $H_V$. In Section \ref{subsecobs}, we conclude this section by computing
an observability constant of $H_{rV}$ from $(\gamma,\sigma)$-distributed sets
and then bound this constant when $r$ varies over $]0,1[$ or $[1,+\infty)$.

Section \ref{wellposed} is devoted to well-posedness of the considered equations and also introduces some
notation for the last section.

Finally, we prove Theorem~\ref{th:main}, Proposition \ref{th:BGS} and their $y$-periodic counterparts
in Section \ref{sec2d} by establishing the corresponding observability
inequalities. In the simplest cases, we also provide an explicit
form of the observability constants.

\medskip

Throughout the paper, we will write $A(x)\lesssim B(x)$ to say that there is a constant $C$ that does not
depend on the parameter $x$ such that $A\leq CB$. At times, it will be more convenient to make
the constants appear explicitely and in this case, they may change from line to line and still be called
with the same letter.

\section{Eigenfunctions of the Schrödinger operator}\label{sec3d}

Let $V$ be a real valued, non-negative function on $\R^n$ with $V\in L^\infty_{\mathrm{loc}}(\R^n)$ and
$$
\lim_{|x|\to+\infty}V(x)=+\infty
$$
but for some integer $m$,
$$
\lim_{|x|\to+\infty}x^mV(x)=0
$$
and consider the associated Schr\"odinger operator
$$
H_Vf(x)=-\Delta f(x)+V(x)f(x).
$$
This operator is well defined on $\mathcal{S}(\R^n)$ and extends to an (unbounded) self-adjoint operator on $L^2(\R^n)$. Further it is positive since, for every $f\in \mathcal{S}(\R^n)$,
$$
\scal{H_Vf,f}=\int_{\R^n}|\nabla f(x)|^2+V(x)|f(x)|^2\,\mbox{d}x\geq 0
$$

The following is well-known ({\it see e.g.} \cite{BS}):

\begin{theorem}
Let $c,\beta>0$ and $V\in L^\infty_{\mathrm{loc}}(\R^n)$ be such that $V(x)\geq c|x|^\beta$.
Then there exists a sequence $(\lambda_k)_{k\in\N}$, with $0\leq \lambda_0\leq\lambda_1\leq\cdots$
and $|\lambda_k|\to+\infty$
and an orthonormal basis $(\phi_k)_{k\in\N}$ of $L^2(\R^n)$ consisting of eigenvectors of $H_V$,
$\phi_k\in H^1(\R^n)$ and $H_V\phi_k=\lambda_k\phi_k$.
Moreover, there exists $a,C>0$ such that, for every $x$,
$$
|\phi_k(x)|\leq C\exp(-a|x|^{1+\frac{\beta}{2}}).
$$
\end{theorem}

\begin{notation}
We will write $\mathcal{E}_V(\lambda)=\mathcal{E}_{H_V}(\lambda)$ for the spectral set associated
to $H_V$, that is
$$
\mathcal{E}_V(\lambda)=\mathrm{span}\{\phi_k\,: k\mbox{ s.t. }\lambda_k\leq\lambda\}.
$$
\end{notation}

In the remaining of this section, we assume that the potential $V\in L^\infty_{\mathrm{loc}}(\R^n)$ satisfies
a global estimate
$$
V(x)\geq c|x|^{\beta}
$$
which is common to Assumptions \ref{A1} and \ref{A2}. We first compute a lower bound of the lowest eigenvalue 
$\lambda_0(V)$. We then give a detailed decay estimate of linear combinations of
eigenfunctions for the Schrödinger operator in terms of the parameters $c,\beta$. 

\subsection{Lower bound of the first eigenvalue}\label{subsec3.1}

We start with the estimate of the first eigenvalue.

\begin{proposition}\label{prp2.2c}
Let $V\in L^1_{\mathrm{loc}}(\R^n)$ and assume that there are $c,\beta>0$ such that  $V\ge c|x|^{\beta}$.
Let $\lambda_0(V)$ be the lowest eigenvalue of the operator $H_V=-\Delta_x+V(x)$. Then we have 
	\begin{equation}
	\label{eq:prp2.2c}
		\lambda_0(V)\ge \mu_{*}:=c^{\frac{2}{\beta+2}} \lambda_{*}
	\end{equation}
	where $\lambda_{*}$ is a positive constant and depends only $\beta$ and $n$.
\end{proposition}

The proof is based on the following result by Barnes, Brascamp and Lieb: 

\begin{theorem}[Barnes, Brascamp and Lieb, \cite{barnes1976lower}]
For all $a>0$, define
\begin{equation*}
	I_V(a)=\int_{\R^{n}}e^{-a V(x)}\d x
\end{equation*}
and assume that, for every $a>0$, $I_V(a)<+\infty$. Then we have
	\begin{equation*}
		\lambda_0(V)\ge \sup_{t>0}t\left[ n +\frac{n}{2}\ln \frac{\pi}{t}-\ln I_V\left( \frac{1}{t} \right)  \right].
	\end{equation*}
\end{theorem}

Now we finish the proof of Proposition~\ref{prp2.2c}.

\begin{proof}[Proof of Proposition \ref{prp2.2c}]
In the case of $V(x)=c|x|^{\beta}$, a change into polar coordinates and then a change of variable $s=acr^{\beta}$
shows that 
\[
\begin{aligned}
I_{c|x|^{\beta}}(a)=& \int_{\R^{n}}e^{-ac|x|^{\beta}}\d x
=\sigma_n\int_0^{+\infty}e^{-acr^{\beta}}r^{n-1}\d r
=\sigma _n\int_0^{+\infty}e^{-s}\left(\frac{s}{ac}\right)^{\frac{n}{\beta}}\frac{\mbox{d}s}{\beta s}
=\frac{\sigma _n}{\beta (a c) ^{n / \beta}}\Gamma\left( \frac{n}{\beta} \right) 
\end{aligned}
\] 
where $\displaystyle \sigma _n= \frac{2 \pi^{ n /2}}{\Gamma\left( n /2 \right) }$ is the surface measure of the unit ball in $\R^{n}$, and $\displaystyle\Gamma (z)=\int_0^{\infty}t ^{z-1}e^{-t}\d t$ is the Gamma function. It follows that 
\[
\lambda_0\left( c|x|^{\beta} \right)\ge \sup_{t>0}t\left[ n-\ln \frac{2}{\beta} \frac{\Gamma\left( \frac{n}{\beta} \right) }{\Gamma\left( \frac{n}{2} \right) }-n\left( \frac{1}{\beta}+\frac{1}{2} \right) \ln t + \frac{n}{\beta}\ln c\right].  
\]
The maximum is attained when
\[
n-\ln \frac{2\pi^{\frac{n}{2}}}{\beta}\frac{\Gamma\left( \frac{n}{\beta} \right) }{\Gamma\left( \frac{n}{2} \right) }-n\left( \frac{1}{\beta}+\frac{1}{2} \right) \ln t + \frac{n}{\beta}\ln c =n \left( \frac{1}{\beta}+\frac{1}{2} \right) 
\] 
so that
\[
\lambda_0(c|x|^{\beta}) \ge n \frac{\beta+2}{2\beta}\exp \left( \frac{\beta-2}{\beta+2} \right) \left( \frac{\beta }{2 \pi^{\frac{n}{2}}} \frac{\Gamma\left( \frac{n}{2\pi^{\frac{n}{2}}} \right) }{\Gamma\left( \frac{n}{\beta} \right) } \right) ^{ \frac{2\beta}{n\left( \beta+2 \right) }}c^{\frac{2}{\beta+2}}:= \lambda_* c^{\frac{2}{\beta+2}}.
\]
Note that, by definition, $\lambda_*$ depends only on $\beta$ and $n$.

Finally, if $V(x)\ge c|x|^{\beta}$, it is obvious that $I_V(a)\le I_{c|x|^{\beta}}(a)$. Hence we obtain
\[
\lambda_0(V)\ge \lambda_0(c|x|^{\beta})\ge \lambda_* c^{\frac{2}{\beta+2}}
\]
as claimed.
\end{proof}

From this, we immediately obtain the following:

\begin{corollary}\label{cor:sem}
Let $V\in L^1_{\mathrm{loc}}(\R^n)$ and assume that there are $c,\beta>0$ such that  $V(x)\ge c|x|^{\beta}$
for every $x\in\R^n$.
Let $r>0$ $H_{rV}=-\Delta_x+rV(x)$. Then, for every $f\in L^2(\R^n)$ and every $s,t>0$,
\begin{equation}
\label{eq:semigroup}
\norm{e^{-tH^s_{rV}}f}_{L^2(\R^n)}\leq e^{-t r^{\frac{2s}{\beta+2}}\mu_*}\norm{f}_{L^2(\R^n)}.
\end{equation}
where $\mu_{*}$ is defined in \eqref{eq:prp2.2c} and is thus a positive constant that depends only $c,\beta$ and $n$.
\end{corollary}

\subsection{Localization property of eigenfunctions}\label{subsec2.2h}

In this section we still assume that $V(x)\ge c|x|^{\beta}$ and prove the following localization (or decay) property of eigenfunctions
which is adapted from \cite[Theorem 1.4]{dicke2022spectral} and \cite[Proposition 3]{zhu2023spectral}
but with more precise quantitative estimates.

\begin{proposition}\label{prp2.1}
	Assume that $V\in L^\infty_{\mathrm{loc}}(\R^n)$ is such that $V(x)\ge c|x|^{\beta}$ and let $H_V=-\Delta+V$.
	Let $\lambda>0$ and $\phi\in\mathcal{E}_V(\lambda)$
	then
	\begin{equation*}
	\|\phi\|_{L^2(\R^{n})}\le 2 \|\phi\|_{L^2{\displaystyle(}\mathcal{B}_\rho(0){\displaystyle)} }
	\quad\mbox{and}\quad
		\|\phi\|_{H^1(\R^{n})}\le 2 \|\phi\|_{H^1{\displaystyle(}\mathcal{B}_\rho(0){\displaystyle)} }
	\end{equation*}
	with 
	\begin{equation}\label{2.16}
	\rho= \hat{C}\left(
	\left(\frac{n}{\beta}+\frac{n+2}{2}\right)\log_+\frac{\lambda+1}{c}+\frac{n+2}{2}\log_+c
	+ \left( \frac{\lambda+2}{c} \right) ^{1 /\beta}+1\right) .
	\end{equation}
	and $\hat{C}$ depending only on $n$.
\end{proposition}

Before the proof of Proposition~\ref{prp2.1}, we give several lemmas. 

\begin{lemma}\label{lma2.1}
Let $V\in L^\infty_{\mathrm{loc}}(\R^n)$ be such that $V(x)\geq c|x|^\beta$.
 Let $\phi$ be an eigenvector of
$H_V=-\Delta+V$ with eigenvalue $\lambda_k$ and
let $R_k= \max \left\{ (\lambda_k+2) /c,1\right\} $.
Then we have
	\begin{equation}\label{2.1}
		\|e^{|\cdot | /2}\phi_k\|^2_{L^2(\R^{n})}\le 7e^{R_k^{1/\beta}+1}\|\phi_k\|^2_{L^2\left( \R^{n} \right) }
	\end{equation}
and
	\begin{equation}\label{2.2}
		\|e^{|\cdot | /2}\nabla \phi_k\|^2_{L^2(\R^{n})}\le C e^{R_k^{1/\beta}} \|\phi_k\|^2_{L^2(\R^n)}
	\end{equation}
where $C>0$ is a constant that depends on $n$ only.
\end{lemma}

The first part is \cite[Proposition~2.3]{dicke2022spectral}. For the second part, we first prove a local Caccioppoli inequality:

\begin{lemma}\label{lem:caccioppoli}
Under the notation of Lemma \ref{lma2.1}:
for every $\rho>0$ there exists a positive constant $C(\rho)$ depending on $\rho$ only such that, for every $z\in\R^n$,
	\begin{equation}
		\|\nabla \phi_k\|_{L^2(\mathcal{B}_{\rho}(z))}^2\le C(\rho)\left( 1+\lambda_k \right) \|\phi_k\|^2_{L^2(\mathcal{B}_{2\rho}(z))}.
	\end{equation}
\end{lemma}

\begin{proof}[Proof of Lemma \ref{lem:caccioppoli}]
	Choose a cutoff function $\eta \in C_c^{\infty}\big( \mathcal{B}_{2\rho}(z) \big) $ and $\eta =1$ in $\mathcal{B}_\rho(z)$ and $|\nabla \eta |< \frac{2}{\rho}$. Let $\psi_k= \eta^2 \phi_k$, then
$$
\int_{\mathcal{B}_{2\rho}(z)}\nabla \phi_k(x) \cdot \nabla \psi_k(x)\,\mbox{d}x
=-\int_{\mathcal{B}_{2\rho}(z)}\psi_k(x) \Delta \phi_k(x)\,\mbox{d}x
=-\int_{\mathcal{B}_{2\rho}(z)}\psi_k(x) \bigl( V(x)-\lambda_k \bigr)\phi_k(x)\,\mbox{d}x.
$$
Further
\begin{align}
\int_{\mathcal{B}_{2\rho}(z)}&|\eta(x) \nabla \phi_k(x)|^2\,\mbox{d}x=
\int_{\mathcal{B}_{2\rho}(z)} \eta(x)^2 \nabla \phi_k(x) \cdot \nabla \phi_k(x)\,\mbox{d}x\nonumber\\
&= - \int_{\mathcal{B}_{2\rho}(z)} 2 \eta(x)\bigr (\nabla  \eta(x) \cdot \nabla \phi_k(x)\bigr) \phi_k(x)\,\mbox{d}x
- \int_{\mathcal{B}_{2\rho}(z)}\eta(x)^2 \bigl( V(x)-\lambda_k \bigr) \phi_k(x)^2\,\mbox{d}x\nonumber\\
&\leq\int_{\mathcal{B}_{2\rho}(z)}\frac{4}{\rho} |\eta(x)\nabla \phi_k(x)| \, |\phi_k(x)|\,\mbox{d}x
-\int_{\mathcal{B}_{2\rho}(z)}V(x)|\eta(x)\phi_k(x)|^2\,\mbox{d}x
+ \lambda_k \|\phi_k\|_{L^2(\mathcal{B}_{2\rho}(z))}^2\nonumber\\
&\leq\int_{\mathcal{B}_{2\rho}(z)}\frac{4}{\rho} |\eta(x)\nabla \phi_k(x)| \, |\phi_k(x)|\,\mbox{d}x
+\lambda_k \|\phi_k\|_{L^2(\mathcal{B}_{2\rho}(z))}^2\label{eq:cacci1}
\end{align}
since $V\geq 0$.
Note that 
\begin{eqnarray*}
\int_{\mathcal{B}_{2\rho}(z)}\frac{4}{\rho}|\eta(x) \nabla \phi_k(x)|\, |\phi_k(x)|\,\mbox{d}x
&=&\int_{\mathcal{B}_{2\rho}(z)}2\abs{\frac{1}{\sqrt{2}}\eta(x) \nabla \phi_k(x)}\, 
\abs{\frac{2\sqrt{2}}{\rho}\phi_k(x)}\,\mbox{d}x\\
&\le& \frac{1}{2}\int_{\mathcal{B}_{2\rho}(z)}|\eta(x) \nabla \phi_k(x)|^2\,\mbox{d}x
+\frac{8}{\rho^2}\int_{\mathcal{B}_{2\rho}(z)}|\phi_k(x)|^2\,\mbox{d}x
\end{eqnarray*}
so that \eqref{eq:cacci1} implies
\begin{eqnarray*}
\int_{\mathcal{B}_{2\rho}(z)}|\eta(x) \nabla \phi_k(x)|^2\,\mbox{d}x
&\leq& \frac{1}{2}\int_{\mathcal{B}_{2\rho}(z)}|\eta(x) \nabla \phi_k(x)|^2\,\mbox{d}x
+\frac{8}{\rho^2}\|\phi_k\|_{L^2{\displaystyle(}\mathcal{B}_{2\rho}(z){\displaystyle)}} ^2
+\lambda_k \|\phi_k\|_{L^2{\displaystyle(}\mathcal{B}_{2\rho}(z){\displaystyle)}} ^2\\
&\leq& \left(1+\frac{8}{\rho^2}\right)(1+\lambda_k)\|\phi_k\|_{L^2{\displaystyle(}\mathcal{B}_{2\rho}(z){\displaystyle)}} ^2.
\end{eqnarray*}
As $\eta=1$ on $\mathcal{B}_{\rho}(z)$ we conclude that
$$
\|\nabla \phi_k\|^2_{L^2{\displaystyle(}\mathcal{B}_{\rho}(z){\displaystyle)}}\le C(\rho)(1+\lambda_k)\|\phi_k\|^2_{L^2{\displaystyle(}\mathcal{B}_{2\rho}(z){\displaystyle)}}
$$
with $C(\rho)=1+\dfrac{8}{\rho^2}$.
\end{proof}

We can now complete the proof of Lemma \ref{lma2.1}.

\begin{proof}[Proof of \eqref{2.2} in Lemma \ref{lma2.1}] 
First we fix $z\in\R^n$ and apply Lemma \ref{lma2.1} with $\rho=1$ to obtain
$$
\|\nabla \phi_k\|^2_{L^2{\displaystyle(}\mathcal{B}_{1}(z){\displaystyle)}}
\le A(1+\lambda_k)\|\phi_k\|^2_{L^2{\displaystyle(}\mathcal{B}_{2}(z){\displaystyle)}}
$$
where $A=C(1)$ is a numerical constant ($A=9$ with the previous proof). Further, as for $x\in \mathcal{B}_{2}(z)$,
$$
e^{-1}e^{|z|/2}\leq e^{|x|/2}\leq ee^{|z|/2}
$$
we get
\begin{eqnarray*}
\|e^{|\cdot|/2}\nabla \phi_k\|^2_{L^2{\displaystyle(}\mathcal{B}_{1}(z){\displaystyle)}}
&\leq& e^2e^{|z|}\|\nabla \phi_k\|^2_{L^2{\displaystyle(}\mathcal{B}_{1}(z){\displaystyle)}}
\le e^2Ae^{|z|} (1+\lambda_k)\|\phi_k\|^2_{L^2{\displaystyle(}\mathcal{B}_{2}(z){\displaystyle)}}\\
&\leq& e^4A (1+\lambda_k)\|e^{|\cdot|/2}\phi_k\|^2_{L^2{\displaystyle(}\mathcal{B}_{2}(z){\displaystyle)}}.
\end{eqnarray*}
We then cover $\R^n$ with a family of balls $\R^n=\bigcup_{i\in\N}\mathcal{B}_{1}(z_i)$
such that the balls $\{\mathcal{B}_{2}(z_i)\}$ have a finite covering number {\it i.e.}
such that $N:=\max_{z\in\R^n}\#\{i\,:z\in\mathcal{B}_{2}(z_i)\}<+\infty$. Then
\begin{eqnarray*}
\|e^{|\cdot|/2}\nabla \phi_k\|^2_{L^2(\R^n)}
&\leq& \sum_{i\in\N}\|e^{|\cdot|/2}\nabla \phi_k\|^2_{L^2{\displaystyle(}\mathcal{B}_{1}(z){\displaystyle)}}
\leq e^4A (1+\lambda_k)\sum_{i\in\N}\|e^{|\cdot|/2}\phi_k\|^2_{L^2{\displaystyle(}\mathcal{B}_{2}(z){\displaystyle)}}\\
&\leq&e^4AN (1+\lambda_k)\|e^{|\cdot|/2}\phi_k\|^2_{L^2(\R^n)}\\
&\leq&e^4AN (1+\lambda_k)7e^{R_k^{1/\beta}+1}\|\phi_k\|^2_{L^2(\R^n)}
\end{eqnarray*}
with \eqref{2.1}. As $N$ depends on $n$ only, we obtain \eqref{2.2} with $C=e^5AN$.
\end{proof}

Define
\begin{equation}
	N(\lambda):=\# \left\{\lambda_k\lvert \lambda_k\le \lambda\right\}. 
\end{equation}
Note that, from
\begin{equation}
	N(\lambda)\le \sum_{k=1}^{N(\lambda)} (\lambda+1-\lambda_k)
\end{equation}
and the lower bound $V(x)\ge c|x|^{\beta}$, the right hand side can be estimated explicitly by means of the classic Lieb-Thirring bound from \cite[Theorem 1]{lieb2001inequalities}. More precisely, for $\lambda >0$
we have
\begin{equation*}
	\begin{aligned}
		\sum_{k=1}^{N(\lambda)} (\lambda+1-\lambda_k)&\lesssim_n \int_{\R^{n}}\max \left\{\lambda+1-V(x),0\right\} ^{n /2+1}\d x\\
							     &\le \int_{\mathcal{B}_0\left( ((\lambda+1) /c)^{1 /\beta}  \right) }(\lambda+1)^{n /2+1}\d x\\
							     &\leq \kappa_n \frac{(\lambda+1)^{\frac{n}{\beta}+\frac{n+2}{2}}}{c^{\frac{n}{\beta}}}
	\end{aligned} 
\end{equation*}
with $\kappa_n$ depending on $n$ only. It follows that
\begin{equation}\label{2.8}
N(\lambda)\leq \kappa_n c^{\frac{n+2}{2}}\left(\frac{\lambda+1}{c}\right)^{\frac{n}{\beta}+\frac{n+2}{2}}
\end{equation}

We are now in position to prove the main result of this section.

\begin{proof}[Proof of Proposition \ref{prp2.1}]
Let us write 
\begin{equation}
\label{eq2.8}
\phi=\sum_{k\leq N(\lambda)}\psi_k
\quad\mbox{ where }\psi_k=c_k\phi_k,\ c_k\in\C.
\end{equation}

For every $\rho>0$, we have
\begin{equation}\label{2.8b}
\begin{aligned}
\|\phi\|^2_{H^{1}{\displaystyle(} \R^{n}\backslash \mathcal{B}_\rho(0) {\displaystyle)} }
&= \|\phi\|^2_{L^2{\displaystyle(}\R^{n}\backslash \mathcal{B}_\rho(0) {\displaystyle)} }
+\|\nabla \phi\|^2_{L^2{\displaystyle(} \R^{n}\backslash \mathcal{B}_\rho(0) {\displaystyle)} }\\
&\le e^{-\rho}\left( \|e^{|\cdot | /2}\phi\|^2_{L^2( \R^{n}) }
+\|e^{|\cdot | /2}\nabla \phi\|^2_{L^2( \R^{n}) } \right). 
		\end{aligned} 
	\end{equation}

Moreover, using the expansion \eqref{eq2.8} and Cauchy-Schwartz, we obtain
	\begin{equation}\label{2.8a}
		\|e^{|\cdot | /2}\phi\|^2_{L^2(\R^{n})}\le \left( \sum_{k=1}^{N(\lambda)} \|e^{|\cdot | /2}\psi_k\|_{L^2(\R^{n})} \right) ^2
	\le N(\lambda) \sum_{k=1}^{N(\lambda)} \|e^{|\cdot | /2}\psi_k\|^2_{L^2\left( \R^{n} \right) }
	\end{equation}
as well as
\begin{equation}\label{2.9}
	\|e^{|\cdot | /2}\nabla \phi\|^2_{L^2\left( \R^{n} \right) }\le N(\lambda)\sum_{k=1}^{N(\lambda)} \|e^{|\cdot | /2}\nabla \phi_k\|^2_{L^2(\R^{n})}.
\end{equation}
Taking \eqref{2.8a} and \eqref{2.9} into \eqref{2.8b}, we obtain
\begin{equation}\label{2.11}
	\|\phi\|^2_{H^{1}(\R^{n}\backslash  \mathcal{B}_\rho(0))} \le e^{-\rho}N(\lambda) \left( \sum_{k=1}^{N(\lambda)} \|e^{|\cdot | /2}\psi_k\|^2_{L^2(\R^{n})}+\sum_{k=1}^{N(\lambda)} \|e^{|\cdot |/2}\nabla \psi_k\|^2_{L^2(\R^{n})}\right). 
\end{equation}
Taking \eqref{2.1}, \eqref{2.2} and \eqref{2.8} into \eqref{2.11} we obtain
$$
\|\phi\|^2_{H^{1}(\R^{n}\backslash \mathcal{B}_\rho(0))}
\le C e^{-\rho}c^{\frac{n+2}{2}}\left(\frac{\lambda+1}{c}\right)^{\frac{n}{\beta}+\frac{n+2}{2}} \sum_{k=1}^{N(\lambda)}e^{R_k^{1/\beta}} \|\psi_k\|^2_{L^2(\R^{n})} 
$$
where $R_k= \max \left\{ (\lambda_k+2) /c,1\right\} $ and $C$ is a constant depending only on $n$.
As the $\psi_k$'s are orthogonal and $R_k \leq R:=\max \left\{ (\lambda+2) /c,1\right\} $, we get
$$
		\|\phi\|^2_{H^{1}(\R^{n}\backslash \mathcal{B}_\rho(0))}\le C e^{-\rho} 
		\left(\frac{\lambda+1}{c}\right)^{\frac{n}{\beta}+\frac{n+2}{2}} c^{\frac{n+2}{2}}e^{R^{1/\beta}} \|\phi\|^2_{L^2(\R^{n})} 
$$
We now chose $\hat C$ large enough so that
\begin{equation*}
	\rho= \hat{C}\left(
	\left(\frac{n}{\beta}+\frac{n+2}{2}\right)\log_+\frac{\lambda+1}{c}+\frac{n+2}{2}\log_+c
	+ \left( \frac{\lambda+2}{c} \right) ^{1 /\beta}+1\right) 
\end{equation*}
satisfies
\begin{equation}\label{2.13}
	\rho \ge \log 2 + \log C + \left(\frac{n}{\beta}+\frac{n+2}{2}\right)\log \frac{\lambda+1}{c}
	+,\frac{n+2}{2}\log_{+}c+R^{1/\beta}+1.
\end{equation}
Note that $\hat C$ depends on $C$ only so that it depends only on $n$. With this choice,
\begin{eqnarray*}
\|\phi\|^2_{L^2{\displaystyle(}\R^{n}\backslash \mathcal{B}_\rho(0){\displaystyle)}}
+\frac{1}{2}\|\nabla \phi\|^2_{L^2{\displaystyle(}\R^{n}\backslash \mathcal{B}_\rho(0){\displaystyle)}}
&\leq&\|\phi\|^2_{H^{1}{\displaystyle(}\R^{n}\backslash \mathcal{B}_\rho(0){\displaystyle)}}\\
&\le& \frac{1}{2} \|\phi\|^2_{L^2(\R^{n})} 
=\frac{1}{2} \|\phi\|^2_{L^2{\displaystyle(}\R^{n}\backslash \mathcal{B}_\rho(0){\displaystyle)}}
+\frac{1}{2} \|\phi\|^2_{L^2{\displaystyle(}\mathcal{B}_\rho(0){\displaystyle)}}.
\end{eqnarray*}
This yelds 
$$
\|\phi\|^2_{L^2{\displaystyle(}\R^{n}\backslash \mathcal{B}_\rho(0){\displaystyle)}}
\leq\|\phi\|^2_{H^{1}{\displaystyle(}\R^{n}\backslash \mathcal{B}_\rho(0){\displaystyle)}}
\le\|\phi\|^2_{L^2{\displaystyle(}\mathcal{B}_\rho(0){\displaystyle)}}
$$
and finally $\|\phi\|^2_{L^2(\R^{n})}\le 2\|\phi\|^2_{L^2{\displaystyle(}\mathcal{B}_\rho(0){\displaystyle)}}$ as well as
$\|\phi\|^2_{H^{1}(\R^{n})}\le 2\|\phi\|^2_{H^1{\displaystyle(}\mathcal{B}_\rho(0){\displaystyle)}}$.
\end{proof}

\section{Spectral inequality for the Schrödinger operator}\label{sec4d}

\subsection{Precised form of Zhu-Zhuge's spectral inequality}

The aim of this section is to prove the following precised form of \cite[Theorem 1]{zhu2023spectral} :

\begin{theorem}[Zhu-Zhuge, precised form]\label{thm1.1}
Let $V\in L^\infty_{\mathrm{loc}}(\R^n)$ be such that  Assumption \ref{A2} is satisfied.
Let $\sigma\geq 0$, $\gamma\in (0, 1 /2)$ and $\omega\subset\R^n$ be a $(\gamma,\sigma)$-distributed set.
Then there exists a constant $C$ depending only on $n$ such that
for every $\lambda\geq\lambda_0(V)$ and every $\phi \in  \mathcal{E}_V(\lambda)$,
\begin{equation}\label{1.11d}
\|\phi\|_{L^2(\R^{n})}\le \left( \frac{1}{\gamma} \right) ^{C \mathcal{J}_V(\lambda)}\|\phi\|_{L^2(\omega)},
\end{equation}
where
\begin{equation}
\begin{aligned}
&\mathcal{J}_V(\lambda):=\mathcal{J}(c_1,c_2,\lambda)=
\mathbf{J}(c_1,\lambda)^{2\sigma /\beta_2}\widehat{\mathbf{J}}(c_1,c_2,\lambda),\\
&\mbox{with}\ \  \mathbf{J}(c_1,\lambda)
= \left(1+\left( \frac{\lambda+2}{c_1} \right)^{1 / \beta_1} +\frac{n+2}{2}\log_{+}c_1  \right)^{\frac{\beta_2}{2}}
\ \mbox{and}\ \ \widehat{\mathbf{J}}(c_1,c_2,\lambda)=\lambda ^{\frac{1}{2}}+c_2^{\frac{1}{2}}\mathbf{J}(c_1,\lambda).
\end{aligned}
\label{eq:zhuzhuge}
\end{equation}
\end{theorem}

The main difference with  Zhu \& Zhuge's result is that the exponent $\mathcal{J}$
depends explicitely on the parameters $c_1$ and $c_2$ in Assumption \ref{A2}. Recall that those are given by
$V=V_1+V_2\in L^\infty_{\mathrm{loc}}(\R^n)$ and
$$
V(x)\geq c_1|x|^{\beta_1}\quad,\quad |V_1(x)|+|\nabla V_1(x)|+|V_2(x)|^{4/3}\leq c_2(|x|+1)^{\beta_2}.
$$

To obtain this result we follow step by step the proof in \cite{zhu2023spectral}. This starts
with two kinds of three-ball inequalities that are already given in a quantitative form sufficient for our needs.
Then, we follow the strategy in \cite{dicke2022spectral,zhu2023spectral} to prove the spectral inequality.

\medskip

Let us start with some notation:

\begin{notation}
Given $L>0$, recall that $\qc_L:= \left[ -\dfrac{L}{2},\dfrac{L}{2} \right] ^{n}$,
$\mathcal{B}_{r}(x) \subset \R^{n}$ is the ball with radius $r$ and center $x$. 
We denote by $\mathbb{B}_r(x)$ the ball in $\R^{n+1}$ centered at $x$ and radius $r$.

Let $\delta \in (0,\frac{1}{2})$, $b=(0,\cdots ,0,-b_{n+1}) \in \R^{n+1}$ and $b_{n+1}= \frac{\delta}{100}$. Define
\begin{equation*}
		W_1=\left\{y \in \R^{n+1}_+\,:\ |y-b|\le \frac{1}{4}\delta\right\}
		\quad,\quad
		W_2= \left\{y \in \R^{n+1}_+\,:\ |y-b|\le \frac{2}{3}\delta\right\},
\end{equation*}
so that $W_1\subset W_2\subset \mathbb{B}_{\delta}(b)$. Write $ \Lambda_L:=\qc_L \cap \Z^{n}$
and define
\[
W_j(z_i):=(z_i,0)+W_j, \quad  j=1,2.
\]
as well as
\[
P_j(L)= \bigcup_{i\in \Lambda_{L}} W_j\left( z_i \right), \quad j=1,2\quad \text{ and } D_{\delta}(L)=\bigcup_{i\in \Lambda_L} \mathcal{B}_{\delta}(z_i).
\] 
Define $R=9 \sqrt{n} $ and
\begin{equation}
	X_1=\qc_L\times [-1,1] \text{ and }\widetilde{X}_{R}=\qc_{L+R}\times [-R,R].
\end{equation}
\end{notation}

The first three-ball inequality in $\R^{n+1}$ we need is the following:

\begin{lemma}[{\cite[Lemma~1]{zhu2023spectral}}]\label{lma3.1}
Let  $\delta \in (0,\frac{1}{2})$. There exist $0<\alpha <1$ and $C>0$, depending only on $n$ such that,
if $v$ is the solution of 
$$
\left\{\begin{matrix}
-(\Delta_x+\partial_{x_{n+1}}^2)v(x,x_{n+1})+V(x)v(x,x_{n+1})=0&x\in\R^n,\ x_{n+1}\in\R\\
v(x,0)=0&x\in\R^n
\end{matrix}\right.
$$
then
	\begin{equation}\label{3.2}
		\big\|v\big\|_{H^{1}{\dst(} P_1(L) {\dst)} }\le \delta^{-\alpha }\exp \left( C\left( 1+\mathcal{G}(V_1,V_2,9\sqrt{n} L \right)  \right) \big\|v\big\|^{\alpha }_{H^{1}{\dst(} P_2(L) {\dst)} }\norm{\frac{\partial v}{\partial y_{n+1}}}^{1-\alpha }_{L^2{\dst(} D_\delta(L) {\dst)} }, 
	\end{equation}
	where 
	\begin{equation}\label{3.3}
		\mathcal{G}\left( V_1,V_2,L \right) =\|V_1\|^{\frac{1}{2}}_{W^{1,\infty}(\qc_L)}+\|V_2\|^{\frac{2}{3}}_{L^{\infty}(\qc_L)}.
	\end{equation}
\end{lemma}

\begin{remark}
Under Assumption \ref{A2}, $\mathcal{G}\left( V_1,V_2,L \right)\leq c_2^{1/2}(1+L)^{\beta_2/2}$.
\end{remark}

The second inequality is:

\begin{lemma}[{\cite[Lemma~2]{zhu2023spectral}}]\label{lma3.2}
	Let $\delta \in (0,\frac{1}{2})$. There exist $C>0$ depending only on $n$, $0<\alpha <1$ depending on $\delta$ and $n$ such that, if $v$ is the solution of 
$$
\left\{\begin{matrix}
-(\Delta_x+\partial_{x_{n+1}}^2)v(x,x_{n+1})+V(x)v(x,x_{n+1})=0&x\in\R^n,\ x_{n+1}\in\R\\
v(x,-y)=-v(x,y)&x\in\R^n,\ y\in\R
\end{matrix}\right.
$$
then	
\begin{equation}\label{3.5}
\big\|v\big\|_{H^{1}(X_1)}
\le \delta^{-2\alpha _1}\exp\Bigl( C\bigl( 1+\mathcal{G}\left( V_1,V_2,9\sqrt{n} L \right)  \bigr)  \Bigr) \big\|v\big\|^{1-\alpha _1}_{H^{1}\left( \widetilde{X}_{R} \right) }\big\|v\big\|^{\alpha_1}_{H^{1}{\dst(} P_1(L){\dst)} },
	\end{equation}
	where $\mathcal{G}(V_1,V_2,L)$ is given by \eqref{3.3}. Further, $\alpha_1$ can be given in the form
\begin{equation}\label{3.7}
	0<\alpha_1= \frac{\epsilon_1}{|\log \delta|+\epsilon_2}<1
\end{equation}
with positive constants $\epsilon_1$ and $\epsilon_2$ depending only on  $n$.
\end{lemma}

Now let $\phi \in \mathcal{E}_\lambda(H)$ and define
\begin{equation}\label{3.8}
	\Phi(x,x_{n+1})=\sum_{0<\lambda_k\le \lambda} \alpha_k \phi_k(x) \frac{\sinh (\sqrt{\lambda_k} x_{n+1})}{\sqrt{\lambda_k} }.
\end{equation}
Then $\Phi(x,x_{n+1})$ satisfies the equation 
\begin{equation}\label{3.9}
	H\Phi:=-\Delta \Phi+V(x)\Phi=0,\quad (x,x_{n+1}) \in \R^{n+1}.
\end{equation}
We need to mention that $\Delta=\Delta_x+\partial_{n+1}^2$ with $\Delta_x=\dst\sum_{j=1}^{n} \partial^2_{j}$
in \eqref{3.9}.
It is easy to check that $\partial_{n+1}\Phi(x,0)=\phi(x)$ and $\Phi(x,0)=0$.

The following estimate for $\Phi$ is standard and can be found in \cite[Lemma~3]{zhu2023spectral}.

\begin{lemma}\label{lma3.3}
Let $\phi \in  \mathcal{E}_V(\lambda)$ and $\Phi$ be given in \eqref{3.9}. For any $\rho >0$, we have
\begin{equation}
2\rho \big\|\phi\big\|^2_{L^2(\R^{n})}
\le \big\|\Phi\big\|^2_{H^{1}\big( \R^{n}\times (-\rho,\rho)\big) }
\le 2\rho \left( 1+ \frac{\rho^2}{3}(1+\lambda)e^{2\rho \sqrt{\lambda} } \right) \big\|\phi\big\|^2_{L^2\big(\R^{n}\big)}.
	\end{equation}
\end{lemma}

To use Proposition \ref{prp2.1}, we need to extend it from $\phi$ to $\Phi$. Indeed, we have the following corollary:

\begin{corollary}\label{crc3.4}
Given the same condition as in Proposition \ref{prp2.1}, we have
\begin{equation}\label{3.11}
\|\Phi\|^2_{H^{1}{\dst(} \R^{n}\times (-1,1) {\dst)} }\le 2\|\Phi\|^2_{H^{1}{\dst(}\mathcal{B}_r \times (-1,1){\dst)}}.
\end{equation}
\end{corollary}

\begin{proof}
	Since $\Phi(\cdot ,x_{n+1})\in \mathcal{E}_{\lambda}(H)$, by Proposition \ref{prp2.1} we obtain
	\begin{equation}\label{3.12}
\|\Phi(\cdot ,x_{n+1})\|^2_{H^{1}(\R^{n})}
\le  2 \|\Phi(\cdot ,x_{n+1})\|^2_{H^1{\dst(} \mathcal{B}_r(0) {\dst)} }.
	\end{equation}
	Since $\partial_{n+1}\Phi(\cdot ,x_{n+1})\in \mathcal{E}_{\lambda}(H)$, we obtain
	\begin{equation}\label{3.13}
		\|\partial_{n+1}\Phi(\cdot ,x_{n+1})\|^2_{L^2(\R^{n})}
		\le \|\partial_{n+1}\Phi(\cdot ,x_{n+1})\|^2_{H^{1}(\R^{n})}
		\le 2 \|\partial_{n+1}\Phi\|^2_{L^2{\dst(} \mathcal{B}_r(0) {\dst)} }.
	\end{equation}
Then we have
\begin{equation*}
	\begin{aligned}
\big\|\Phi\big\|^2_{H^{1}{\dst(}\R^{n}\times (-1,1){\dst)}}
=& \int_{-1}^{1}\big\|\Phi(\cdot ,x_{n+1})\big\|^2_{L^2(\R^{n})}\d x_{n+1}
		+ \int_{-1}^{1}\sum_{j=1}^{n} \big\|\partial_j \Phi(\cdot ,x_{n+1})\big\|^2_{L^2(\R^{n})}\d x_{n+1}\\
		 &+ \int_{-1}^{1}\big\|\partial_{n+1}\Phi(\cdot ,x_{n+1})\big\|^2_{L^2(\R^{n})}\d x_{n+1}\\
\le& 
\int_{-1}^{1}\big\|\Phi(\cdot ,x_{n+1})\big\|^2_{H^{1}(\R^{n})}\d x_{n+1} 
+ \int_{-1}^{1}2\big\|\partial_{n+1}\Phi(\cdot ,x_{n+1})\big\|^2_{L^2{\dst(}\mathcal{B}_r(0){\dst)}}\d x_{n+1}\\
	\end{aligned} 
\end{equation*}
with \eqref{3.13}. Using \eqref{3.12} we then obtain
\begin{equation*}
	\begin{aligned}
\big\|\Phi\big\|^2_{H^{1}{\dst(}\R^{n}\times (-1,1){\dst)}}\le 
& \int_{-1}^{1}2 \big\|\Phi(\cdot ,x_{n+1})\big\|^2_{H^{1}{\dst(}\mathcal{B}_r(0){\dst)}}\d x_{n+1}
+\int_{-1}^{1}2\big\|\partial_{n+1}\Phi(\cdot ,x_{n+1})\big\|^2_{L^2{\dst(}\mathcal{B}_r(0){\dst)}}\d x_{n+1}\\
		=& 2\big\|\Phi\big\|^2_{H^{1}{\dst(} \mathcal{B}_r(0){\dst)} }
	\end{aligned} 
\end{equation*}
as claimed.
\end{proof}

We can now prove Theorem~\ref{thm1.1}:

\begin{proof}[Proof of Theorem \ref{thm1.1}]
	Let $L=2\left\lceil r \right\rceil +1$, where $r$ is given in Proposition \ref{prp2.1} and $\left\lceil a \right\rceil $ means the largest integer smaller than $a+1$.  Then we have $\mathcal{B}_r(0)\subset \qc_L$.
Moreover, we can decompose  $\qc_L$ as
	\begin{equation}
		\qc_L= \bigcup_{k\in \Lambda_L } \left( k+\left[ -\frac{1}{2},\frac{1}{2} \right]^{n}  \right). 
	\end{equation}
	For each $k \in \Lambda_L$, we have $|k|\le \sqrt{n} \left\lceil r \right\rceil $. As $\gamma\in (0,\frac{1}{2})$,
	we get
	\begin{equation}
		\delta:= \gamma ^{1+ \left( \sqrt{n} \left\lceil r \right\rceil  \right) ^{\sigma } }\le \gamma^{1+|k|^{\sigma }}, \quad  \forall k \in \Lambda_L \cap \Z^{n}.
	\end{equation}

Now we show an interpolation inequality. Note that $\Phi$ is odd in $x_{n+1}$, so taking $v=\Phi$,	
we combine \eqref{3.2} in Lemma~\ref{lma3.1} and \eqref{3.5} in Lemma~\ref{lma3.2} with $\delta$ and $L$ 
defined above to get
\begin{equation*}
\begin{aligned}
\|\Phi\|_{H^{1}(X_1)}&
\le \delta^{-2 \alpha_1}\exp\Big( C\big( 1+\mathcal{G}( V_1,V_2,9 \sqrt{n} L)\big)\Big) 
\big\|\Phi\big\|^{1-\alpha_1}_{H^{1}(\widetilde{X}_{R})}
\big\|\Phi\big\|^{\alpha_1}_{H^{1}{\dst(} P_1(L){\dst)} }\\
&\le \delta ^{-2 \alpha_1-\alpha \alpha_1}\exp\Big( C\big( 1+\mathcal{G}( V_1,V_2,9 \sqrt{n} L)\big)\Big) 
\big\|\Phi\big\|^{\alpha \alpha_1}_{H^{1}{\dst(}P_2(L){\dst)}}
\left\|\frac{\partial \Phi}{\partial y_{n+1}}\right\|^{\alpha_1  (1-\alpha )}_{L^2{\dst(}D_\delta(L){\dst)}}
\big\|\Phi\big\|^{1-\alpha_1}_{H^{1}(\widetilde{X}_{R})}\\
& \le \delta^{-3 \alpha_1}\exp\Big( C\big( 1+\mathcal{G}( V_1,V_2,9 \sqrt{n} L)\big)\Big) 
\big\|\phi\big\|^{\hat{\alpha}}_{L^2\big(D_\delta(L)\big)}
\big\|\Phi\big\|^{1-\hat{\alpha}}_{H^{1}\left( \widetilde{X}_{R} \right) },
		\end{aligned} 
	\end{equation*}
where $\hat{\alpha}=\alpha_1 (1-\alpha )$ and we have used the facts $P_2(L) \subset \widetilde{X}_{R}$ and $\dfrac{\partial \Phi}{\partial y_{n+1}}(\cdot ,0)=\phi$. Here and below, the symbol $C$ may represent different positive constants depending on $n$.

Recall $\alpha_1$ in \eqref{3.7}, we have $\alpha_1 \approx \hat{\alpha}\approx \frac{1}{|\log \delta|}$ for any $\delta \in (0, \frac{1}{2})$. Hence $\delta ^{-3 \alpha_1}\le C$ and then
	\begin{equation}\label{3.18}
		\big\|\Phi\big\|_{H^{1}\left( X_1 \right) }
		\le \exp \Big( C\big( 1+\mathcal{G}( V_1,V_2,9 \sqrt{n} L) \big) \Big) 
		\big\|\phi\big\|^{\hat{\alpha}}_{L^2\left( \omega \cap \qc_L \right) }\big\|\Phi\big\|^{1-\hat{\alpha}}_{H^{1}(\widetilde{X}_{R})},
	\end{equation}
where we have also used the fact $D_\delta (L)\subset \omega \cap \qc_L$.

Substituting $L=2\left\lceil r \right\rceil +1$ and \eqref{2.16} into $\mathcal{G}(V_1,V_2,L)$ and by Assumption \ref{A2}, we have
\begin{eqnarray*}
	\mathcal{G}(V_1,V_2,9 \sqrt{n} L)&\lesssim& c_2^{\frac{1}{2}} \left( 2\left\lceil r \right\rceil +2 \right) ^{\frac{\beta_2}{2}}\lesssim c_2^{\frac{1}{2}} \left( \frac{n+4}{2\beta_1}\log_+ \frac{\lambda+1}{c_1}+\left( \frac{\lambda+2}{c_1} \right) ^{1 /\beta_1}+\frac{n+2}{2}\log_+c_1 +1 \right) ^{\frac{\beta_2}{2}}\\
&\lesssim&c_2^{\frac{1}{2}} \left(+\left( \frac{\lambda+2}{c_1} \right) ^{1 /\beta_1}+\frac{n+2}{2}\log_{+}c_1 +1 \right) ^{\frac{\beta_2}{2}}	:=c_2^{\frac{1}{2}}\mathbf{J}(c_1,\lambda)
\end{eqnarray*}
where $\mathbf{J}(c_1,\lambda)$ was defined in \eqref{eq:zhuzhuge}.
We can then write \eqref{3.18} as 
\begin{equation}\label{3.20}
	\big\|\Phi\big\|_{H^{1}\left( X_1 \right) }
	\le \exp \left( C c_2^{\frac{1}{2}} \mathbf{J}(c_1,\lambda) \right) \big\|\phi\big\|^{\hat{\alpha}}_{L^2\left( \omega \cap \qc_L \right) }\big\|\Phi\big\|^{1- \hat{\alpha}}_{H^{1}\left( \widetilde{X}_{R} \right) }. 
\end{equation}

We now bound $\big\|\Phi\big\|^2_{H^{1}{\dst(} \R^{n}\times (-\rho,\rho){\dst)}}$ from above and below
by respectively taking $\rho =R$ and $\rho =1$ in Lemma \ref{lma3.3}. This gives
\begin{equation}
\frac{\big\|\Phi\big\|^2_{H^{1}{\dst(}\R^{n}\times(-R,R){\dst)}}}{\big\|\Phi\big\|^2_{H^{1}{\dst(}\R^{n}\times(-1,1) {\dst)}}}
\le R \left( 1+\frac{R^2}{3}\left( 1+\lambda \right)  \right) \exp \left( 2 R \sqrt{\lambda}  \right) \le \exp \left( C_2 \sqrt{\lambda}  \right). 
\end{equation}
With the aid of \eqref{3.11} and  $\mathcal{B}_r(0)\subset \qc_L$, we get
\begin{equation}\label{3.22}
	\begin{aligned}
\big\|\Phi\big\|_{H^{1}{\dst(} \R^{n}\times (-R,R) {\dst)} } 
&\le \exp\left( \frac{1}{2}C_2 \sqrt{\lambda}  \right) \big\|\Phi\big\|_{H^{1}{\dst(} \R^{n}\times (-1,1) {\dst)} }\\
&\le \sqrt{2} \exp \left( \frac{1}{2}C_2 \sqrt{\lambda}  \right)
\big\|\Phi\big\|_{H^{1}{\dst(}\qc_L\times (-1,1) {\dst)} }.
	\end{aligned} 
\end{equation}
Recall that $X_1=\qc_L\times (-1,1)$, substituting \eqref{3.22} into \eqref{3.20} we obtain
\begin{equation}
	\big\|\Phi\big\|_{H^{1}{\dst(} \R^{n}\times (-R,R){\dst)} }
	\le \exp \big( C_3 \widehat{\mathbf{J}}(c_1,c_2,\lambda) \big) \big\|\phi\big\|^{\hat{\alpha}}_{L^2\left( \omega \cap \qc_L \right) }\big\|\Phi\big\|^{1- \hat{\alpha}}_{H^{1}\left( \widetilde{X}_{R} \right) }   
\end{equation}
where $\widehat{\mathbf{J}}(c_1,c_2,\lambda)$ was defined in \eqref{eq:zhuzhuge}.
Since $\widetilde{X}_{R}\subset \R^{n}\times (-R,R)$, it follows that
\begin{equation}
	\big\|\Phi\big\|_{H^{1}{\dst(} \R^{n}\times (-R,R) {\dst)} }
	\le \exp \left( \hat{\alpha}^{-1}C_3\widehat{\mathbf{J}}\left( c_1,c_2,\lambda \right)  \right)
	\big \|\phi\big\|_{L^2\left( \omega \cap \qc_L \right) }. 
\end{equation}
Recall that
\begin{equation}
\hat{\alpha}^{-1}\approx \alpha^{-1}_1\approx |\log \delta| \approx |\log \gamma|\mathbf{J} ^{\frac{2\sigma }{\beta_2}}
\end{equation}
we obtain
\begin{equation}
	\big\|\Phi\big\|_{H^{1}\lb \R^{n}\times (-R,R){\dst)} }\le \left( \frac{1}{\gamma} \right) ^{C \mathbf{J}^{\frac{2\sigma }{ \beta_2}}\widehat{\mathbf{J}}}\big\|\phi\big\|_{L^2\left( \omega\cap \qc_L \right) }.
\end{equation}
Finally, using the lower bound in Lemma \ref{lma3.3} with $\rho =R$, we obtain
\begin{equation}
\big\|\phi\big\|_{L^2(\R^{n})}
\le \left( \frac{1}{2 R} \right) ^{\frac{1}{2}}\big\|\Phi\big\|_{H^{1}{\lb}\R^{n}\times (-R,R){\rb}}
\le \left( \frac{1}{\gamma} \right) ^{C \mathbf{J}^{\frac{2\sigma }{ \beta_2}}\widehat{\mathbf{J}}}\big\|\phi\big\|_{L^2( \omega ) }
\end{equation}
where $C$ is a positive constant depending only on $n$.
\end{proof}

\subsection{Scaling the potential}\label{subsec.5}
The aim of this section is to consider the ``scaled''
Schrödinger operator
\begin{equation}\label{1.13a}
	H_{rV}=-\Delta_x+rV(x), \quad r>0
\end{equation}
and to evaluate the influence of $r$ on the exponent $\mathcal{J}$
in Theorem \ref{thm1.1}. This will be slightly different according to when $V$
satisfies  Assumption \ref{A1} or \ref{A2}.

First, we assume that $V$ in \eqref{1.13a} satisfies Assumption~\ref{A1}. Then the potential $rV$ satisfies Assumption~\ref{A1} but replacing  $c_1,c_2$ with $rc_1,rc_2$ respectively, while $\beta_1,\beta_2$ are unchanged. 
We thus need to estimate
$$
\mathcal{J}^1(r,\lambda):=\mathcal{J}_{rV}=\mathcal{J}(rc_1,rc_2,\lambda)
$$
for $\lambda\geq\lambda_0(rV)$.

On the other hand, if we assume that $V$ in \eqref{1.13a} satisfies Assumption~\ref{A1}. Then the potential $rV$ satisfies Assumption~\ref{A1} by replacing  $c_1,c_2$ with $rc_1,r^{4/3}c_2$ respectively, and the same $\beta_1,\beta_2$. 
We thus need to estimate
$$
\mathcal{J}^2(r,\lambda):=\mathcal{J}(rc_1,r^{4/3}c_2,\lambda)
$$
for $\lambda\geq\lambda_0(rV)$.

Further, according to Proposition \ref{prp2.2c}, $\lambda_0(V)\geq c_1^{\frac{2}{\beta_1+2}}\lambda_*$ 
where $\lambda_*$ depends on $\beta_1$ and $n$ only. Thus 
$\lambda_0(rV)\geq r^{\frac{2}{\beta_1+2}} c_1^{\frac{2}{\beta_1+2}}\lambda_*
:=r^{\frac{2}{\beta_1+2}}\mu_*$. Recall that $\mu_*$ depends only on $c_1,\beta_1,n$.

The estimates of $\mathcal{J}^1,\mathcal{J}^2$ we need are given in the following proposition:

\begin{proposition}\label{prop:powers}
Fix $\sigma\geq 0$, $\gamma\in (0, 1 /2)$ and let $\omega\subset\R^n$ be a $(\gamma,\sigma)$-distributed set.
Fix $c_1,c_2,\beta_1,\beta_2>0$ and let $V\in L^\infty_{\mathrm{loc}}(\R^n)$ be such that one of Assumptions 
\ref{A1}-\ref{A2} is satisfied with parameters $c_1,c_2,\beta_1,\beta_2$. Let
$\zeta=\dfrac{\beta_2+2\sigma}{2\beta_1}$ and $\varepsilon>0$.

Then there exists a constant $C_{V,\omega}\geq 1$ depending on $n,c_1,c_2,\beta_1,\beta_2,\sigma$ and $\varepsilon$ such that,
for every $r>0$ and for every $\phi\in\mathcal{E}_{rV}(\lambda)$ 
\begin{equation}\label{1.11dr}
\|\phi\|_{L^2(\R^{n})}\le \left( \frac{1}{\gamma} \right) ^{C_{V,\omega} \mathcal{J}_{rV}(\lambda)}
\|\phi\|_{L^2(\omega)},
\end{equation}
with
\begin{equation}
\label{eq:expr1}
\mathcal{J}_{rV}(\lambda)=
\begin{cases}
r^{a_-}+r^{b_-}\lambda^{\zeta}&\mbox{if }0<r<1\\
r^{a_++\varepsilon}+r^{b_++\varepsilon}\lambda^{\zeta}&\mbox{if }r\geq 1
\end{cases}
\end{equation}
where $a_-,b_-,a_+,b_+$ are given as follows:
\begin{enumerate}
\renewcommand{\theenumi}{\roman{enumi}}
\item if $V$ satisfies Assumption \ref{A1}, then
\begin{center}
\begin{tabular}{ll}
$a_{-}=\dfrac{1}{2}-\zeta\leq 0$,&$b_-=\dfrac{1}{2}-\zeta \le 0$,\\[12pt]
$a_{+}=\dfrac{1}{2}$,&$b_+=\begin{cases}
	\dfrac{1-\sigma-2\zeta}{\beta_1+2} \le 0 & \mbox{if } (\beta_1-\beta_2)\sigma=0\\
	\max \left( \dfrac{1}{2}-\dfrac{\sigma}{\beta_1+2},\dfrac{1}{\beta_1+2} \right) - \dfrac{2\zeta}{\beta_1+2} & \mbox{if } (\beta_1-\beta_2)\sigma\neq 0
\end{cases}$.
\end{tabular}
\end{center}
\item If $V$ satisfies Assumption \ref{A2} with $\beta_*:=3\beta_2- 4\beta_1-2\leq 0$, then
\begin{center}
\begin{tabular}{ll}
$a_-=\dfrac{2}{3}-\zeta$,&$b_-=\begin{cases}
	-\left(\dfrac{\sigma}{\beta_1}+\dfrac{\beta_2-\beta_1}{(\beta_1+2)\beta_1}\right)\le 0&\mbox{if }\beta_*\sigma=0\\
	\min\left( \dfrac{2}{3}-\dfrac{\beta_2}{2\beta_1}-\dfrac{\sigma}{\beta_1+2},\dfrac{1}{\beta_1+2}-\dfrac{\sigma}{\beta_1} \right) -\dfrac{2\zeta}{\beta_1+2}<0&\mbox{if }\beta_*\sigma\neq0
\end{cases}$,\\[12pt]
$a_+=\dfrac{2}{3}$,&$b_+=-\dfrac{\sigma}{\beta_1+2}+\dfrac{2}{3}-\dfrac{\beta_2}{2\beta_1}-\dfrac{2\zeta}{\beta_1+2}$.
\end{tabular}
\end{center}
\item If $V$ satisfies Assumption \ref{A2} with $\beta_*>0$, then
\begin{center}
\begin{tabular}{ll}
$a_{-}=\dfrac{2}{3}-\zeta$,&
$b_-=-\left(\dfrac{\beta_*}{2(\beta_1+2)}+\dfrac{\sigma}{\beta_1+2}+\dfrac{2\zeta-1}{\beta_1+2}\right)<0$,
\\[12pt]
$a_{+}=\dfrac{2}{3}$,&$b_+=\begin{cases}
	-\dfrac{\beta_2-\beta_1}{(\beta_1+2)\beta_1}\le 0 & \mbox{if } \sigma=0\\
	\max\left(\dfrac{2}{3} - \dfrac{\sigma}{\beta_1+2},\dfrac{1}{\beta_1+2} \right) -\dfrac{2\zeta}{\beta_1+2} & \mbox{if }\sigma\neq 0
\end{cases}$.
\end{tabular}
\end{center}
\end{enumerate}
\end{proposition}

\begin{remark}\label{rem:r=1}
The separation of cases at $r=1$ in \eqref{eq:expr1} is arbitrary and can be replaced
with 
$$\mathcal{J}_{rV}(\lambda)=
\begin{cases}
r^{a_-}+r^{b_-}\lambda^{\zeta}&\mbox{if }0<r<\alpha\\
r^{a_++\varepsilon}+r^{b_++\varepsilon}\lambda^{\zeta}&\mbox{if }r\geq \alpha
\end{cases}
$$
for any $\alpha>0$. This only influences the constant $C_{V,\omega}$ in \eqref{1.11dr}
but not the exponents $a_\pm$, $b_\pm$ and $\zeta$.
\end{remark}

\begin{remark}\label{rq:2pot}
Note that if $V,\tilde V$ both satisfy either Assumption \ref{A1} or Assumption \ref{A2}
with same exponents $\beta_1,\beta_2$ and (without loss of generality) same
coefficients $c_1,c_2$, then $rV+\tilde V$ satisfy the same assumption,
still with exponents $\beta_1,\beta_2$, but with parameters
$(r+1)c_1,(r+1)c_2$ in case of Assumption \ref{A1} and $(r+1)c_1,(r^{4/3}+1)c_2$.
As $(r^{4/3}+1)\lesssim (1+r)^{4/3}$, we get that
$\mathcal{J}_{rV+\tilde V}$ satisfies the same bounds but with $r$ replaced by $r+1\geq 1$.
\end{remark}

\begin{proof}
Recall that $\mathcal{J}(c_1,c_2,\lambda)$ is given in \eqref{eq:zhuzhuge}. 
As $\lambda_0(rV)\geq r^{\frac{2}{\beta_1+2}}\mu_*$ we may write
$\lambda=r^{\frac{2}{\beta_1+2}}\mu$ with $\mu\gtrsim 1$. We thus want to estimate 
$\mathcal{J}^1(r,\lambda)=\mathcal{J}(rc_1,rc_2,r^{\frac{2}{\beta_1+2}}\mu)$.

As $\mu\gtrsim 1$, 
We will use that, for $a,b,c,d\in\R$, $\alpha,\beta\ge 0$
\begin{eqnarray}
(r^a+r^b\mu^\alpha)(r^c+r^d\mu^\beta)&=&r^{a+c}+r^{a+d}\mu^\beta+r^{b+c}\mu^\alpha+
r^{b+d}\mu^{\alpha+\beta}\nonumber\\
&\lesssim &
\begin{cases}
r^{a+c}+r^{\min(a+d,b+c,b+d)}\mu^{\alpha+\beta}&\mbox{when }0<r<1\\
r^{a+c}+r^{\max(a+d,b+c,b+d)}\mu^{\alpha+\beta}&\mbox{when }r\geq 1
\end{cases}\label{z.1}
\end{eqnarray}
when $\delta:=(b-a)\beta- (d-c)\alpha \neq 0$.
Here we use that $\mu\gtrsim 1$ to absorb $\mu^\alpha,\mu^\beta$ into $\mu^{\alpha+\beta}$
and then keep only the smallest power of $r$ when $r<1$ and the largest one otherwise.

When $\delta=(b-a)\beta-(d-c)\alpha=0$, one can improve this as follows:
observe that for $t\geq 0$, $t^\alpha,t^\beta\leq\max(1,t^{\alpha+\beta})$ so that $(1+t^\alpha)(1+t^\beta)\leq 3(1+t^{\alpha+\beta})$.
Thus, with $t=r^{\frac{b-a}{\alpha}}\mu=r^{\frac{d-c}{\beta}}\mu$
we get
\begin{eqnarray}
	(r^{a}+r^{b}\mu^{\alpha})(r^{c}+r^{d}\mu^{\beta})
	&=&r^{a+c}\left(1+\left(r^{\frac{b-a}{\alpha}}\mu\right)^\alpha\right)\left(1+\left(r^{\frac{d-c}{\beta}}\mu\right)^\beta\right)\nonumber\\
	&\lesssim& r^{a+c}\left(1+\left(r^{\frac{b-a}{\alpha}}\mu\right)^{\alpha+\beta}\right)\label{z.2bis}\\
	&\lesssim& r^{a+c}+r^{b+d}\mu^{\alpha+\beta}.\label{z.2}
\end{eqnarray}


Now
$$
\mathbf{J}(rc_1,r^{\frac{2}{\beta_1+2}}\mu)\lesssim
\left( \frac{r^{\frac{2}{\beta_1+2}}\mu+1}{rc_1} \right)^{\frac{\beta_2}{2\beta_1}}+\left(\log_{+}rc_1\right)^{\frac{\beta_2}{2}}+1
\lesssim 1+(\log_+r)^{\frac{\beta_2}{2}}+r^{-\frac{\beta_2}{2\beta_1}}+r^{-\frac{\beta_2}{2(\beta_1+2)}}\mu^{\frac{\beta_2}{2\beta_1}}
$$
with the implied constants depending on $\beta_1,\beta_2$ and $c_1$, as well as
\begin{align}
\mathbf{J}(rc_1,r^{\frac{2}{\beta_1+2}}\mu)^{2\sigma /\beta_2}&\lesssim
\left( \frac{r^{\frac{2}{\beta_1+2}}\mu+1}{rc_1} \right)^{\frac{\sigma}{\beta_1}}+(\log_+ r)^{\sigma}+1
\lesssim 1+(\log_+ r)^{\sigma}+r^{-\frac{\sigma}{\beta_1}}+r^{-\frac{\sigma}{\beta_1+2}}\mu^{\frac{\sigma}{\beta_1}}\nonumber\\
&\lesssim
\begin{cases}
r^{-\frac{\sigma}{\beta_1}}\bigl(1+(r^{\frac{2}{\beta_1+2}}\mu)^{\frac{\sigma}{\beta_1}}\bigr)&\mbox{if }0<r<1\\
\left(1+r^{-\frac{\sigma}{\beta_1+2}}\mu^{\frac{\sigma}{\beta_1}}\right)\log^{\sigma}(r+1)&\mbox{if }r\geq1
\end{cases}\label{estJbld}
\end{align}
with the implied constants depending on $\sigma,\beta_1,\beta_2$ and $c_1$.

On the other hand
\begin{eqnarray*}
\widehat{\mathbf{J}}(rc_1,rc_2,r^{\frac{2}{\beta_1+2}}\mu)
&=&r^{\frac{1}{\beta_1+2}}\mu^{1/2}+r^{1/2}c_2^{1/2}\mathbf{J}(rc_1,r^{\frac{2}{\beta_1+2}}\mu)\\
&\lesssim&r^{\frac{1}{\beta_1+2}}\mu^{1/2}+r^{1/2}+r^{1/2}(\log_+ r)^{\frac{\beta_2}{2}}+r^{-\frac{\beta_2-\beta_1}{2\beta_1}}+
r^{\frac{1}{2}-\frac{\beta_2}{2(\beta_1+2)}}\mu^{\frac{\beta_2}{2\beta_1}}\\
&\lesssim&r^{1/2}+r^{1/2}(\log_+ r)^{\frac{\beta_2}{2}}+r^{-\frac{\beta_2-\beta_1}{2\beta_1}}+
\left(r^{\frac{1}{\beta_1+2}}+r^{\frac{1}{2}-\frac{\beta_2}{2(\beta_1+2)}}\right)\mu^{\frac{\beta_2}{2\beta_1}}
\end{eqnarray*}
where, in the last line, we used that $\mu\gtrsim 1$ and $\dfrac{\beta_2}{2\beta_1}\geq\dfrac{1}{2}$
so that $\mu^{\frac{1}{2}}\lesssim\mu^{\frac{\beta_2}{2\beta_1}}$.
The implied constant here depends on $\sigma,\beta_1,\beta_2$ and $c_1,c_2$. Now note that
$\beta_2\geq\beta_1$ implies
$$
\frac{1}{\beta_1+2}\geq \frac{2-(\beta_2-\beta_1)}{2(\beta_1+2)}=\frac{1}{2}-\frac{\beta_2}{2(\beta_1+2)}.
$$
It follows that
\begin{equation}
\widehat{\mathbf{J}}(rc_1,rc_2,r^{\frac{2}{\beta_1+2}}\mu)
\lesssim
\begin{cases}
r^{-\frac{\beta_2-\beta_1}{2\beta_1}}\bigl(1+(r^{\frac{2}{\beta_1+2}}
\mu)^{\frac{\beta_2}{2\beta_1}}\bigr)&\mbox{if }0<r<1\\
\left(r^{\frac{1}{2}}+
r^{\frac{1}{\beta_1+2}}\mu^{\frac{\beta_2}{2\beta_1}}+r^{\frac{1}{2}}\right)\log^{\frac{\beta_2}{2}}(r+1)&\mbox{if }r\geq1
\end{cases}.\label{estJbld2}
\end{equation}

For $0<r<1$, we multiply \eqref{estJbld} with  \eqref{estJbld2} by using \eqref{z.2bis} and obtain  
\begin{equation}
	\mathcal{J}_{rV}(r,\lambda)\lesssim r^{-\frac{\sigma}{\beta_1}-\frac{\beta_2-\beta_1}{2\beta_1}}(1+(r^{\frac{2}{\beta_1+2}}\mu)^{\frac{\beta_2+2\sigma}{2\beta_1}}), \quad 0<r<1.\label{z.3}
\end{equation}
For $r\ge 1$, we take $a=0,b=-\frac{\sigma}{\beta_1+2},c=\frac{1}{2},d=\frac{1}{\beta_1+2},\alpha=\frac{\sigma}{\beta_1},\beta=\frac{\beta_2}{2\beta_1}$ and then obtain
 \[
\delta=(b-a)\beta-(d-c)\alpha= \frac{(\beta_1-\beta_2)\sigma}{2(\beta_1+2)\beta_1}.
\] 
Then $\delta=0$ if and only if $\beta_1=\beta_2$ or $\sigma=0$. 
Multiplying \eqref{estJbld} and \eqref{estJbld2} and using \eqref{z.2} when $\delta= 0$
(resp. \eqref{z.1} when $\delta\neq0$) we obtain
\begin{equation}
	\mathcal{J}_{rV}(r,\lambda)\lesssim \begin{cases}
		\left(r^{\frac{1}{2}}+r^{-\frac{\sigma}{\beta_1+2}+\frac{1}{\beta_1+2}}\mu^{\frac{\beta_2+2\sigma}{2\beta_1}}\right)\log^{\sigma+\frac{\beta_2}{2}}(r+1) & \mbox{if } (\beta_1-\beta_2)\sigma=0\\
		\left(r^{\frac{1}{2}}+r^{\widetilde{b}_{+}}\mu^{\frac{2\sigma+\beta_2}{2\beta_1}}\right)\log^{\sigma+\frac{\beta_2}{2}}(r+1)& \mbox{if } (\beta_1-\beta_2)\sigma\neq 0
	\end{cases},\quad r\ge 1\label{z.4}
\end{equation}
where
\[
\widetilde{b}_{+}:=\max\left( \frac{1}{2}- \frac{\sigma}{\beta_1+2},\frac{1}{\beta_1+2} \right)\ge \frac{1}{\beta_1+2}-\frac{\sigma}{\beta_1+2}. 
\]
Replacing $\mu$ in \eqref{z.3} and \eqref{z.4} by $r^{-\frac{2}{\beta_1+2}}\lambda$, and absorbing the $\log$ term in the power with the help of $\varepsilon$ we obtain the claimed inequality.

\medskip

To estimate $\mathcal{J}_{rV}$ under Assumption \ref{A2}, the main difference is that we need to estimate
\begin{eqnarray*}
\widehat{\mathbf{J}}(rc_1,r^{4/3}c_2,r^{\frac{2}{\beta_1+2}}\mu)
&=&r^{\frac{1}{\beta_1+2}}\mu^{1/2}+r^{2/3}c_2^{1/2}\mathbf{J}(rc_1,r^{\frac{2}{\beta_1+2}}\mu)\\
&\lesssim&r^{\frac{2}{3}}+r^{\frac{2}{3}-\frac{\beta_2}{2\beta_1}}+r^{\frac{1}{\beta_1+2}}\mu^{1/2}+
r^{\frac{2}{3}-\frac{\beta_2}{2(\beta_1+2)}}\mu^{\frac{\beta_2}{2\beta_1}}+r^{\frac{2}{3}} (\log_+r)^{\frac{\beta_2}{2}}\\
&\lesssim&r^{\frac{2}{3}}+r^{\frac{2}{3}-\frac{\beta_2}{2\beta_1}}+\left(r^{\frac{1}{\beta_1+2}}
+r^{\frac{2}{3}-\frac{\beta_2}{2(\beta_1+2)}}\right)\mu^{\frac{\beta_2}{2\beta_1}}++r^{\frac{2}{3}} (\log_+r)^{\frac{\beta_2}{2}}.
\end{eqnarray*}
There are now two cases to be distinguished since 
$$
\frac{1}{\beta_1+2}\leq\frac{2}{3}-\frac{\beta_2}{2(\beta_1+2)}
$$
if and only if $3\beta_2\leq 4\beta_1+2$ or equivalently $\dfrac{1}{2}\leq \dfrac{\beta_2}{2\beta_1}\leq\dfrac{2}{3}+\dfrac{1}{3\beta_1}$.

So assume first that $3\beta_2\leq 4\beta_1+2$, then
\begin{equation}
\widehat{\mathbf{J}}(rc_1,r^{4/3}c_2,r^{\frac{2}{\beta_1+2}}\mu)\lesssim
\begin{cases}
r^{\frac{2}{3}-\frac{\beta_2}{2\beta_1}}+r^{\frac{1}{\beta_1+2}}
\mu^{\frac{\beta_2}{2\beta_1}}
&\mbox{if }0< r< 1\\
\left(r^{\frac{2}{3}}
+r^{\frac{2}{3}-\frac{\beta_2}{2(\beta_1+2)}}\mu^{\frac{\beta_2}{2\beta_1}}\right)\log^{\frac{\beta_2}{2}}(r+1)
&\mbox{if }r\geq 1
\end{cases}.\label{estJbld3}
\end{equation}
For $0<r<1$, we take $a=-\dfrac{\sigma}{\beta_1}$, $b=-\dfrac{\sigma}{\beta_1+2}$,
$c=\dfrac{2}{3}-\dfrac{\beta_2}{2\beta_1}$, $d=\dfrac{1}{\beta_1+2}$,
$\alpha=\dfrac{\sigma}{\beta_1}$ and $\beta=\dfrac{\beta_2}{2\beta_1}$ so that
\[
\delta:=(b-a)\beta-(d-c)\alpha)=\frac{(2+4\beta_1-3\beta_2)\sigma}{6(\beta_1+2)\beta_1}.
\]
Hence $\delta\neq 0$ if and only if $3 \beta_2< 4 \beta_1+2$ and $\sigma\neq 0$. Using \eqref{z.1} and  \eqref{z.2} when $\delta\neq 0$ and $\delta=0$ respectively to multiply \eqref{estJbld} and \eqref{estJbld3}, we obtain
\begin{equation}
	\mathcal{J}_{rV}(r,\lambda)\lesssim \begin{cases}
		r^{\frac{2}{3}-\frac{\beta_2+2\sigma}{2\beta_1}}+r^{\frac{1-\sigma}{\beta_1+2}}\mu^{\frac{\beta_2+2\sigma}{2\beta_1}} & \mbox{if } (2+4\beta_1-3\beta_2)\sigma= 0\\
		r^{\frac{2}{3}-\frac{\beta_2+2\sigma}{2\beta_1}}+r^{\widetilde{b}_{-}}\mu^{\frac{\beta_2+2\sigma}{2\beta_1}} & \mbox{if } (2+4\beta_1-3\beta_2)\sigma\neq 0
	\end{cases}, \quad 0<r<1\label{z.5}
\end{equation}
where
$$
\tilde b_-:=\min\left(
\frac{2}{3}-\frac{\beta_2}{2\beta_1}-\frac{\sigma}{\beta_1+2},
\frac{1}{\beta_1+2}-\frac{\sigma}{\beta_1}\right)\leq
\frac{1}{\beta_1+2}-\frac{\sigma}{\beta_1+2}.
$$
For $r\ge 1$, we take  $a=0$, $b=-\dfrac{\sigma}{\beta_1+2}$, $c=\dfrac{2}{3}$, $d=\dfrac{2}{3}-\dfrac{\beta_2}{2\beta_1}$, $\alpha=\dfrac{\sigma}{\beta_1}$, $\beta=\dfrac{\beta_2}{2\beta_1}$ and then obtain
\[
\delta:=(b-a)\beta-(d-c)\alpha)=0.
\]
Using \eqref{z.2} to multiply \eqref{estJbld} and \eqref{estJbld3}, we obtain 
\begin{equation}
	\mathcal{J}_{rV}(r,\lambda)\lesssim r^{\frac{2}{3}}+r^{-\frac{\sigma}{\beta_1+2}+\frac{2}{3}-\frac{\beta_2}{2\beta_1}}\mu^{\frac{\beta_2+2\sigma}{2\beta_1}}, \quad  r\ge 1.\label{z.6}
\end{equation}
Replacing $\mu$ in \eqref{z.5} and \eqref{z.6} by $r^{-\frac{2}{\beta_1+2}}\lambda$, we obtain the claimed inequality.  

\smallskip

On the other hand, if $3 \beta_2>4\beta_1+2$, then
\begin{equation}
\widehat{\mathbf{J}}(rc_1,r^{4/3}c_2,r^{\frac{2}{\beta_1+2}}\mu)\lesssim
\begin{cases}
r^{\frac{2}{3}-\frac{\beta_2}{2\beta_1}}+r^{\frac{2}{3}-\frac{\beta_2}{2(\beta_1+2)}}
\mu^{\frac{\beta_2}{2\beta_1}}
&\mbox{if }0< r< 1\\
\left(r^{\frac{2}{3}}
+r^{\frac{1}{\beta_1+2}}\mu^{\frac{\beta_2}{2\beta_1}}\right)\log^{\frac{\beta_2}{2}}(r+1)
&\mbox{if }r\geq 1
\end{cases}.\label{estJbld4}
\end{equation}
For $0<r<1$, we take  $a=-\dfrac{\sigma}{\beta_1}$,
$b=-\dfrac{\sigma}{\beta_1+2}$, $c=\dfrac{2}{3}-\dfrac{\beta_2}{2\beta_1}$, $d=\dfrac{2}{3}-\dfrac{\beta_2}{2(\beta_1+2)}$, $\alpha=\dfrac{\sigma}{\beta_1}$, $\beta=\dfrac{\beta_2}{2\beta_1}$ and then obtain
\[
\delta:=(b-a)\beta-(d-c)\alpha=0.
\] 
Using \eqref{z.2} to multiply \eqref{estJbld} and \eqref{estJbld3}, we obtain
\begin{equation}
	\mathcal{J}_{rV}(r,\lambda)\lesssim r^{\frac{2}{3}-\frac{\beta_2+2\sigma}{2\beta_1}}+r^{-\frac{\sigma}{\beta_1+2}+\frac{2}{3}-\frac{\beta_2}{2(\beta_1+2)}}\mu^{\frac{\beta_2+2\sigma}{2\beta_1}},\quad 0<r<1.\label{z.7}
\end{equation}
For $r\ge 1$, we take  $a=0$, $b=-\dfrac{\sigma}{\beta_1+2}$,
$c=\dfrac{2}{3}$, $d=\dfrac{1}{\beta_1+2}$,
$\alpha=\dfrac{\sigma}{\beta_1}$, $\beta=\dfrac{\beta_2}{2\beta_1}$ and then obtain 
\[
\delta:=(b-a)\beta-(d-c)\alpha=-\frac{(3 \beta_2-4 \beta_1-2)\sigma}{6(\beta_1+2)\beta_1}\le 0.
\]
Hence $\delta\neq 0$ if and only if $\sigma\neq 0$. Using \eqref{z.1} and \eqref{z.2} when $\delta\neq 0$ and $\delta=0$ respectively to multiply \eqref{estJbld} and \eqref{estJbld4}, we obtain
\begin{equation}
	\mathcal{J}_{rV}(r,\lambda)\lesssim \begin{cases}
		\left(r^{\frac{2}{3}}+r^{\frac{1}{\beta_1+2}}\mu^{\frac{\beta_2}{2\beta_1}}\right)\log^{\frac{\beta_2}{2}}(r+1)& \mbox{if } \sigma=0\\
		\left(r^{\frac{2}{3}}+r^{\widetilde{b}_{+}}\mu^{\frac{\beta_2+2\sigma}{2\beta_1}}\right)\log^{\sigma+\frac{\beta_2}{2}}(r+1)& \mbox{if } \sigma\neq 0
	\end{cases},
	r\ge 1\label{z.8}
\end{equation}
where
$$
\tilde b_+=\max\left(-\frac{\sigma}{\beta_1+2}+\frac{2}{3},\frac{1}{\beta_1+2}\right)
\geq\frac{1}{\beta_1+2}-\frac{\sigma}{\beta_1+2}
$$
Absorbing the $\log$ term in the power with the help of $\varepsilon$, we again conclude by replacing $\mu$ in \eqref{z.7} and \eqref{z.8}  by $r^{-\frac{2}{\beta_1+2}}\lambda$.
\end{proof}

\subsection{From the spectral inequality to observability of scaled Schr\"odinger operators}
\label{subsecobs}

We can now prove observability inequalities for ``scaled'' Schr\"odinger operators $H_{rV}=-\Delta_x+rV(x)$
on $\R^n$ in which we give an explicit dependence on $r$.
To do so, we will use the following result from \cite{nakic2020sharp} that allows to go
from a spectral inequality to an observability inequality. ({\it see} \cite[Theorem 2.1]{beauchard2018null} for a similar 
result in which constants are less explicit).

\begin{theorem}[{\cite[Theorem 2.8]{nakic2020sharp}}]\label{thm2.1d}
Let $P$ be a non-negative selfadjoint operator on $L^2\left( \R^{n} \right)$, $s>0$ and
$\omega \subset \R^{n}$ be measurable. Suppose that there are $\alpha_0\geq 1$, $\alpha_1\ge 0$ and $0<\zeta <s$ such that, for all $\lambda \ge 0$ and $\phi \in \mathcal{E}_\lambda(P)$,
\begin{equation}\label{2.1c}
\|\phi\|^2_{L^2(\R^{n})}\le \alpha_0 e^{\alpha_1\lambda^{\zeta}}\|\phi\|^2_{L^2(\Omega)}.
\end{equation}
Then there exist positive constants $\kappa_1,\kappa_2,\kappa_3>0$ depending only on $n,\zeta$ and $s$, such that for all $T>0$ and $g \in L^2(\R^{n})$ we have the observability estimate
\[
\|e^{-tP^s}g\|^2_{L^2(\R^{n})}\le \frac{C_{\mathrm{obs}}}{T}\int_0^{T}\|e^{-tP^s}g\|^2_{L^2(\omega)}\d t,
\] 
where the positive constant $C_{\mathrm{obs}}>0$ is given by
\begin{equation}\label{2.2c}
C_{\mathrm{obs}}=\kappa_1 \alpha_0^{\kappa_2}
\exp\left( \kappa_3 \alpha_1^{\frac{s}{s-\zeta}}T^{-\frac{\zeta}{s-\zeta}}  \right).
\end{equation}
\end{theorem}

In \cite{nakic2020sharp}, the theorem is stated with $s=1$, but this form follows directly from
the transformation formula for spectral measures ({\it see} \cite[Proposition~4.24]{schmudgen2012unbounded}), for all $s>0$ and $\lambda \ge 0$ we have
\begin{equation}\label{1f}
	\mathcal{E}_\lambda(P^{s})=\mathcal{E}_{\lambda ^{\frac{1}{s}}}(P).
\end{equation}

We can now prove the following:

\begin{proposition}\label{prp2.2}
Let $r>0$ and $V\in L^1_{\mathrm{loc}}(\R^n)$ satisfying one of Assumptions \ref{A1}-\ref{A2},
with corresponding constants $c_1,c_2,\beta_1,\beta_2$.
Let $H_{rV}=-\Delta+rV$ be the Schr\"odinger operator associated to $rV$. 
Let $0<\gamma<1$ and $\sigma\geq 0$ and $\omega\subset\R^n$ be a $(\gamma,\sigma)$-distributed set.
Let $a_\pm$, $b_\pm$ and $\varepsilon$ be defined in Proposition \ref{prop:powers}.
Let $\zeta=\frac{\beta_2+2\sigma}{2\beta_1}$ and $s>\zeta$.

For $T>0$ and  $g\in L^2(\R^{n})$ we have
\begin{equation*}
\|e^{-TH_{rV}^{s}}g\|^2_{L^2\left( \R^{n} \right) }
\leq \frac{C_{\mathrm{obs}}(T,s,rV,\omega)}{T}\int_0^{T}\|e^{-t H_{r}^{s}}g\|^2_{L^2(\omega)}\d t
\end{equation*}
where the positive constant $C_{\mathrm{obs}}$ is given by
\begin{equation}\label{2.7c}
 \begin{cases}
  C_0\exp\ent{C_1\log\left(\frac{1}{\gamma}\right)r^{a_-}
+C_2T^{-\frac{\zeta}{s-\zeta}}\log^{\frac{s}{s-\zeta}}\left(\dfrac{1}{\gamma}\right)r^{\frac{s}{s-\zeta}b_{-}}
-C_3Tr^{\frac{2s}{\beta_1+2}}} &\mbox{for }  0<r<1\\
C_0\exp\ent{C_1\log\left(\frac{1}{\gamma}\right)r^{a_++\varepsilon}
+C_2T^{-\frac{\zeta}{s-\zeta}}\log^{\frac{s}{s-\zeta}}\left(\dfrac{1}{\gamma}\right)r^{\frac{s}{s-\zeta}(b_{+}+\varepsilon)}
-C_3Tr^{\frac{2s}{\beta_1+2}}}&\mbox{for } r\ge 1
\end{cases} 
\end{equation}
with positive constants $C_0,C_1,C_2,C_3$ depending on $n,\sigma,\varepsilon$ and on $c_1,c_2,\beta_1,\beta_2$.
\end{proposition}

\begin{proof}
For $P=H_{rV}$, Inequality \eqref{2.1c} was established in
Proposition \ref{prop:powers}. Recall that $a_-,b_-$ is for $0<r<1$ and $a_+,b_+$ is for $r\geq 1$. For $0<r<1$,
we can apply Theorem \ref{thm2.1d} with $\alpha_0=C\exp\ent{C_{V,\omega}\log\left(\frac{1}{\gamma}\right)r^{a_-}}\geq 1$ (we can assume that $C\geq 1$), $\alpha_1=C\log\left(\dfrac{1}{\gamma}\right)r^{b_-}$. we thus get, the following:

For $T>0$ and  $g\in L^2(\R^{n})$ we have
\begin{equation*}
\|e^{-\frac{T}{2}H_{r}^{s}}g\|^2_{L^2\left( \R^{n} \right) }
\leq \frac{\tilde C_{\mathrm{obs}}(T/2,s,rV,\omega)}{T}\int_0^{T/2}\|e^{-t H_{r}^{s}}g\|^2_{L^2(\omega)}\d t
\end{equation*}
where the positive constant $\tilde C_{\mathrm{obs}}(T/2,s,rV,\omega)$ is given by
\begin{equation}\label{2.7d}
\tilde C_{\mathrm{obs}}(T/2,s,rV,\omega)
= C_0\exp\ent{C_1\log\left(\frac{1}{\gamma}\right)r^{a_-}
+C_2T^{-\frac{\zeta}{s-\zeta}}\log^{\frac{s}{s-\zeta}}\left(\dfrac{1}{\gamma}\right)r^{\frac{s}{s-\zeta}b_-}}
\end{equation}
with constants $C_0,C_1,C_2>0$ depending on $n,\sigma$ and on $c_1,c_2,\beta_1,\beta_2$ of Assumptions \ref{A1}-\ref{A2} only.

Next, using Corollary \ref{cor:sem}, we get
\begin{equation}\label{2.7e}
\|e^{-TH_{r}^{s}}g\|^2_{L^2\left( \R^{n} \right) }=
\|e^{-\frac{T}{2}H_{r}^{s}}\bigl(e^{-\frac{T}{2}H_{r}^{s}}g\bigr)\|^2_{L^2\left( \R^{n} \right) }
\leq e^{-\frac{\mu_*}{2}r^{\frac{2s}{\beta_1+2}}T}
\|e^{-\frac{T}{2}H_{r}^{s}}g\|^2_{L^2\left( \R^{n} \right) }
\end{equation}
with $\mu_*$ depending only on $c_1,\beta_1$ and $n$.
Combining \eqref{2.7d} and \eqref{2.7e} we obtain \eqref{2.7c} with $C_3=\mu_*/2$
and $C_{\mathrm{obs}}(T,s,rV,\omega)=\tilde C_{\mathrm{obs}}(T/2,s,rV,\omega)\exp\big(-C_3Tr^{\frac{2s}{\beta_1+2}}\big)$.

For $r\ge 1$, the proof is the same by replacing $a_-$ and $b_-$ with $a_++\varepsilon$ and $b_++\varepsilon$, and the constants will also depend on $\varepsilon$. 
\end{proof}

Simple calculus shows the following:

\begin{lemma}\label{lem:prp2.2}
With the notation and conditions of Proposition \ref{prp2.2}
\begin{enumerate}
\renewcommand{\theenumi}{\roman{enumi}}
\item for every $T>0$, $\sup_{0<r<1}C_{\mathrm{obs}}(T,s,rV,\omega)<+\infty$ if and only if $a_-,b_-\geq 0$ in which case we obtain
$$
\sup_{0<r<1}C_{\mathrm{obs}}(T,s,rV,\omega)\leq A_0(T,s,\gamma):=
C_0\left(\frac{1}{\gamma}\right)^{C_2}\exp\ent{
C_2T^{-\frac{\zeta}{s-\zeta}}\log^{\frac{s}{s-\zeta}}\left(\dfrac{1}{\gamma}\right)}.
$$
\item If $\dfrac{2s}{\beta_1+2}>\nu:=\max\left(a_+,\dfrac{s}{s-\zeta}b_+\right)$ then
$\sup_{r\geq 1}C_{\mathrm{obs}}(T,s,rV,\omega)<+\infty$. Further, 
$$
\sup_{r\geq 1}C_{\mathrm{obs}}(T,s,rV,\omega)\leq A_1(T,s,\gamma):=
\begin{cases}C_0\exp C_4 T^{-\left(\delta+\frac{\zeta}{s-\zeta}(1+\delta)\right)}\log^{\frac{s}{s-\zeta}(1+\delta)}\left(\frac{1}{\gamma}\right)&\mbox{for }T\leq 1\\
C_0\exp C_5 T^{-\delta}\log^{\frac{s}{s-\zeta}(1+\delta)}\left(\frac{1}{\gamma}\right)&\mbox{otherwise}
\end{cases}
$$
where $\delta=\left(\dfrac{s}{\beta_1+2}-\dfrac{\nu}{2}\right)^{-1}$ and $C_4,C_5$ do not depend on $T,\gamma$.
\end{enumerate} 
\end{lemma}

\begin{proof}
The first bound is trivial, for the second one, we use that
$$
\sup_{r\geq 1}C_{\mathrm{obs}}(T,s,rV,\omega)
\leq
C_0\sup_{r\geq 1}\exp\ent{C_1\min(1,T)^{-\frac{\zeta}{s-\zeta}}\log^{\frac{s}{s-\zeta}}
\left(\frac{1}{\gamma}\right)r^{\nu+\frac{s}{s-\zeta}\varepsilon}
-C_3Tr^{\frac{2s}{\beta_1+2}}}
$$
where $\varepsilon$ is chosen to be 
\begin{equation}
    0<\varepsilon=\frac{1}{2}\left(\frac{2s}{\beta_1+2}-\nu\right)\frac{s-\zeta}{s}<\left(\frac{2s}{\beta_1+2}-\nu\right) \frac{s-\zeta}{s}.
\end{equation}
and then that $\sup_{r\geq 1}\exp(Ar^u-Br^v)=\exp\dfrac{v-u}{u}A\left(\dfrac{Au}{Bv}\right)^{1/(v-u)}$ when $0< u\leq v$.
\end{proof}

\begin{remark}\label{rem:vtildev}
From Remark \ref{rq:2pot} we get that, 
if $V,\tilde V$ both satisfy Assumption \ref{A1} or Assumption \ref{A2} with same
paramerters $c_1$, $c_2$, $\beta_1$ and $\beta_2$, then 
$$
\sup_{r\geq 0}C_{\mathrm{obs}}(T,s,rV+\tilde V,\omega)
\lesssim\sup_{r\geq 1}C_{\mathrm{obs}}(T,s,rV,\omega).
$$
\end{remark}

\begin{lemma}\label{lem:cond}
With the notation and conditions of Proposition \ref{prp2.2} and assuming that $s>\zeta$,
the condition $\dfrac{2s}{\beta_1+2}>\nu:=\max\left(a_+,\dfrac{s}{s-\zeta}b_+\right)$
(resp. $\dfrac{2s}{\beta_1+2}=\nu$)
simplifies to the following:
\begin{enumerate}
\renewcommand{\theenumi}{\roman{enumi}}
\item if $V$ satisfies Assumption \ref{A1}, 
$s>s_{A1}:=\dfrac{\beta_1+2}{4}$;

\item if $V$ satisfies Assumption \ref{A2}, 
$s>s_{A2}:=\dfrac{\beta_1+2}{3}$.
\end{enumerate}
\end{lemma}

\begin{proof}
	(i) Let us first assume that $V$ satisfies Assumption \ref{A1}. Remember that
	\begin{equation*}
	a_{+}=\dfrac{1}{2},\quad b_+=\begin{cases}
	\dfrac{1-\sigma-2\zeta}{\beta_1+2} \le 0 & \mbox{if } (\beta_1-\beta_2)\sigma=0\\
	\max \left( \dfrac{1}{2}-\dfrac{\sigma}{\beta_1+2},\dfrac{1}{\beta_1+2} \right) - \dfrac{2\zeta}{\beta_1+2} & \mbox{if } (\beta_1-\beta_2)\sigma\neq 0
	\end{cases}.
	\end{equation*}
	In particular,  $\displaystyle \frac{2s}{\beta_1+2}>a_+$ is satisfied if and only if $s>s_{A_1}$.
	We will thus assume that $s>s_{A_1}$ and we only need to prove that
	$\dfrac{2s}{\beta_1+2}>\dfrac{s}{s-\zeta}b_{+}$ as well. As $s>\zeta$, this is equivalent to
	\begin{equation}
		s>\frac{\beta_1+2}{2}b_{+}+\zeta.\label{z.9}
	\end{equation}
	Using again that  $s>\zeta$, this is obviously satisfied if $b_{+}\le 0$, hence we only need to consider the case $b_{+}>0$, \textit{i.e.},
	\[
	b_{+}=\max \left( \frac{1}{2}-\frac{\sigma}{\beta_1+2},\frac{1}{\beta_1+2} \right) -\frac{2\zeta}{\beta_1+2}>0.
	\] 
	As $\zeta\geq\dfrac{1}{2}$, $b_+$ can only be positive when 
	\[
	b_{+}=\frac{1}{2}-\frac{\sigma}{\beta_1+2}-\frac{2\zeta}{\beta_1+2}.
	\] 
	Substitute this into \eqref{z.9} we obtain the equivalent form
	\[
		s> \frac{\beta_1+2}{4}-\frac{\sigma}{2}=s_{A1}-\frac{\sigma}{2}
	\] 
	which is clearly satified under our assumption $s>s_{A1}$.

	\smallskip
	(ii) Let us now assume that $V$ satisfies Assumption \ref{A2} with $3 \beta_2\le 4\beta_1+2$ and $s>s_{A2}$. Remember that
	\begin{equation*}
		a_{+}=\frac{2}{3},\quad b_{+}=\frac{2}{3}-\left(\frac{\sigma}{\beta_1+2}+\frac{\beta_2}{2\beta_1}+\frac{2\zeta}{\beta_1+2}\right)
	\end{equation*}
	Again,  $\displaystyle \frac{2s}{\beta_1+2}>a_+$ is equivalent to $s>s_{A_2}$. We thus assume that $s>s_{A_2}$
	and we only need to prove 
	\eqref{z.9} which, in this case, is equivalent to 
	\[
	s> s_{A_2}-\frac{\beta_1+2}{2}\left(\frac{\sigma}{\beta_1+2}+\frac{\beta_2}{2\beta_1}+\frac{2\zeta}{\beta_1+2}\right)
	\] 
	and is thus clearly satisfied when $s>s_{A2}$.

	\smallskip
	
	Assume that $V$ satisfies Assumption \ref{A2} with  $3\beta_2>4\beta_1+2$.
	Remember that 
	\begin{equation*}
		a_{+}=\dfrac{2}{3}, \quad  b_{+}=\begin{cases}
			\dfrac{1}{\beta_1+2}-\dfrac{\beta_2}{(\beta_1+2)\beta_1}\le 0& \mbox{if } \sigma=0\\
			\max \left( \dfrac{2}{3}-\dfrac{\sigma}{\beta_1+2},\dfrac{1}{\beta_1+2} \right) -\dfrac{2\zeta}{\beta_1+2} &\mbox{if }\sigma\neq 0
		\end{cases}.
	\end{equation*}
	As for the previous case, we have to impose $s>s_{A_2}$ and only need to show that
	$\dfrac{2s}{\beta_1+2}>\dfrac{s}{s-\zeta}b_+$. This is satisfied when $b_+\leq 0$
	so that we only need to consider the case $\sigma>0$
	and
	\[
	b_{+}=\frac{2}{3}-\frac{2\zeta+\sigma}{\beta_1+2}
	\] 
	since $2\zeta\geq 1$.
	Substitute this into \eqref{z.9} we want $s>s_{A2}-(4\zeta+2\sigma)$
	which follows from the assumption $s>s_{A2}$.
\end{proof}

Let us now determine under which conditions 
\begin{equation}
	\sup_{r>0}C_{\mathrm{obs}}(T,s,rV,\omega)<+\infty.\label{z.10}
\end{equation}
By Lemma \ref{lem:prp2.2} (i) it is necessary to ensure
\begin{equation}
	a_{-}\ge 0 \,\,\text{ and }\,\,b_{-}\ge 0.\label{z.11}
\end{equation}

From Proposition \ref{prop:powers}, this only happens in the following cases:

\begin{enumerate}
\renewcommand{\theenumi}{\roman{enumi}}
\item If $V$ satisfies Assumption \ref{A1}, $a_-,b_-\geq0$ is equivalent to
$\zeta:=\frac{\beta_2+2\sigma}{2\beta_1}=\dfrac{1}{2}$ which, as $\beta_2\geq\beta_1$ and $\sigma\geq0$,
reduces to $\beta_2=\beta_1$ and $\sigma=0$.

\item If $V$ satisfies Assumption \ref{A2} and $3\beta_2-4\beta_1-2\leq 0$, then $b_-<0$
unless $\sigma(3\beta_2-4\beta_1-2)=0$ and in this case $b_-=-\left(\dfrac{\sigma}{\beta_1+2}+\dfrac{\beta_2-\beta_1}{\beta_1(\beta_1+2)}\right)<0$ unless $\beta_2=\beta_1$ and $\sigma=0$.
Note that again $\zeta=\dfrac{1}{2}$ in this case.
\end{enumerate}

Note also that if $\beta_1=\beta_2$ then $3\beta_2-4\beta_1-2=-(\beta_1+2)< 0$ so that, if $V$ satisfies Assumption \ref{A2}, $b_-\geq0$ if and only if $\beta_2=\beta_1$ and $\sigma=0$.

Now it remains to make sure that
\begin{equation*}
	\sup_{r\ge 1}C_{\mathrm{obs}}(T,s,rV,\omega)<+\infty.\label{z.12}
\end{equation*}
By Lemma~\ref{lem:cond} it is sufficient to assume that
$s> s_{A1}$ (resp. $s> s_{A2}$) under Assumption~\ref{A1} (resp. Assumption~\ref{A2}). Hence we obtain the following corollary from Lemma~\ref{lem:prp2.2} and Lemma~\ref{lem:cond}.

\begin{corollary}\label{cor:bddA2}
With the notation and conditions of Proposition \ref{prp2.2}, then
$$
\sup_{r>0}C_{\mathrm{obs}}(T,s,rV,\omega)<+\infty
$$
if
\begin{enumerate}
\renewcommand{\theenumi}{\roman{enumi}}
	\item $V$ satisfies Assumption~\ref{A1} with $\beta_1=\beta_2,\sigma=0$ and  $s>s_{A1}$;
	\item $V$ satisfies Assumption~\ref{A2} with $\beta_1=\beta_2,\sigma=0$ and $s>s_{A2}$.
\end{enumerate}
Further if $V$ satisfies Assumption~\ref{A1} (resp.  $V$ satisfies Assumption~\ref{A2})
and the above conditions is satisfied, then set
$s_{A}=s_{A1}$(resp. $s_{A}=s_{A2}$), $\delta=\frac{2(\beta_1+2)}{s-s_A}$, then
\begin{enumerate}
\renewcommand{\theenumi}{\alph{enumi}}
\item if $s>s_{A}$ and $T\leq 1$,
\begin{equation}
\label{eq:b><}
\begin{aligned}
B_-(T,s,rV,\omega):=\sup_{r>0}C_{\mathrm{obs}}(s,rV,\omega)\lesssim\left(\frac{1}{\gamma}\right)^{C_1}\exp&\ent{C_2T^{-\frac{1}{2s-1}}
\log^{-\frac{2s}{2s-1}}\left(\frac{1}{\gamma}\right)}\\
&+\exp\ent{C_3T^{-\delta-\frac{1+\delta}{2s-1}}\log^{-\frac{2s}{2s-1}(1+\delta)}\left(\frac{1}{\gamma}\right)};
\end{aligned}
\end{equation}
\item if $s>s_{A}$  and $T\geq 1$,
\begin{equation}
\label{eq:b>>}
\begin{aligned}
B_+(T,s,rV,\omega):=\sup_{r>0}C_{\mathrm{obs}}(s,rV,\omega)\lesssim\left(\frac{1}{\gamma}\right)^{C_1}\exp&\ent{C_2T^{-\frac{1}{2s-1}}
\log^{-\frac{2s}{2s-1}}\left(\frac{1}{\gamma}\right)}\\
&+\exp\ent{C_3T^{-\delta}\log^{-\frac{2s}{2s-1}(1+\delta)}\left(\frac{1}{\gamma}\right)};
\end{aligned}
\end{equation}
\end{enumerate}

\end{corollary}

\section{Essential self adjointness of Baouendi-Grushin operators}\label{wellposed}

The results and the proofs in this section are essentially the same as for \cite[Proposition 3.1]{DM}.

\medskip

For $f\in L^1(\R^{n}\times\R^m)$, we write its partial Fourier transform as
$$
\fc_2[f](x,\eta)=f^\eta(x)=(2\pi)^{-m/2}\int_{\R^m}f(x,y)e^{-i\scal{y,\eta}}\,\mbox{d}\eta\qquad \eta\in\R^m.
$$
For $f\in  L^1(\R^{n}\times\R^m)\cap L^2(\R^{n}\times\R^m)$, Parseval's relation
writes $\|\fc_2[f]\|_{L^2(\R^{n+m})}=\|f\|_{L^2(\R^{n+m})}$. We
may thus extend $\fc$ into a unitary transform on $L^2(\R^{n}\times\R^m)$
and in particular $f^\eta$ is well defined for almost every $\eta$ and,
for $u\in L^2(\R^n)$, $v\in L^2(\R^m)$,
$$
\int_{\R^{n+m}}f(x,y)\overline{u(x)v(y)}\,\mbox{d}x\,\mbox{d}y
=\int_{\R^{m}}\int_{\R^n}\fc_2[f](x,\eta)\overline{u(x)}\,\mbox{d}x\,\overline{\widehat{v}(\eta)}\,\mbox{d}\eta
$$
where $\widehat{u}$ (resp. $\widehat{v}$) is the usual Fourier transforms of $u$ in $L^2(\R^n)$, resp. of $v$
in $L^2(\R^m)$.

We will use the same notation for $f\in L^1(\R^{n}\times\T^m)$, and its partial Fourier coefficient as
$$
\fc_2^p[f](x,k)=f^k(x)=(2\pi)^{-m/2}\int_{\T^m}f(x,y)e^{-i\scal{y,k}}\,\mbox{d}\eta\qquad k\in\Z^m.
$$
For $f\in  L^1(\R^{n}\times\T^m)\cap L^2(\R^{n}\times\T^m)$, Parseval's relation
writes $\|\fc_2[f]\|_{L^2(\R^{n},\ell^2(\Z^m)}=\|f\|_{L^2(\R^{n}\times\T^m)}$. We
may thus extend $\fc$ into a unitary transform from $L^2(\R^{n}\times\T^m)$
to $L^2(\R^{n},\ell^2(\Z^m)$.
In particular,
for $u\in L^2(\R^n)$, $v\in L^2(\T^m)$,
$$
\int_{\R^{n}\times\T^m}f(x,y)\overline{u(x)v(y)}\,\mbox{d}x\,\mbox{d}y
=\sum_{k\in \Z^{m}}\int_{\R^n}\fc_2[f](x,k)\overline{u(x)}\,\mbox{d}x\,\overline{c_k(v)}
$$
where $\big(c_k(v)\big)_{k\in\Z^m}$ are the usual Fourier coefficients of $v$.

\begin{proposition}
Let $V\in L^1_{\mathrm{loc}}(\R^n)$ and assume that $V$ satisfies assumption \ref{A2}
Let $\mathcal{L}_V=-\Delta_x-V(x)\Delta_y$, then $\mathcal{L}$ is essentially self-adjoint on $L^2(\R^{n}\times\R^m)$
as well as on $L^2(\R^{n}\times\T^m)$.
\end{proposition}

\begin{proof}
Let $\mathcal{L}^*$ be the Hilbert adjoint of $\mathcal{L}$ on $L^2(\R^{m+n})$.
According to \cite[Corollary of Theorem VIII.3]{RS} it is enough to show that, for every $\lambda\in\C\setminus\R$,
$\mathcal{L}^*-\lambda$ is one-to-one. We thus want to show that, if
$f\in L^2(\R^{n+m})$ is such that, for every $\varphi\in\mathcal{C}^\infty_c(\R^{n+m})$,
$$
\int_{\R^{n+m}}f(x,y)\overline{\bigl(-\Delta_x-V(x)\Delta_y-\lambda\bigr)}\varphi(x,y)\,\mbox{d}x\,\mbox{d}y=0
$$
then $f=0$.

Taking $\varphi(x,y)=u(x)v(y)$ with $u,v$ smooth and compactly supported, we assume that
\begin{eqnarray*}
0&=&\int_{\R^{m}}\int_{\R^n}f(x,y)\overline{\bigl(-\Delta u(x)-\lambda u(x)\bigr)}\,\mbox{d}x\,\overline{v(y)}\,\mbox{d}y
-\int_{\R^{m}}\int_{\R^n}f(x,y)V(x)u(x)\,\mbox{d}x\,\overline{\Delta v(y)}\,\mbox{d}y\\
&=&\int_{\R^{m}}\int_{\R^n} \fc_2[f](x,\eta)\overline{\bigl(-\Delta +|\eta|^2V(x)-\lambda\bigr) u(x)}\mbox{d}x\,
\overline{\widehat{v}(\eta)}\,\mbox{d}\eta.
\end{eqnarray*}
This implies that, for every $u\in\mathcal{C}_c(\R^n)$,
$$
\int_{\R^n} \fc_2[f](x,\eta)\overline{\bigl(-\Delta +|\eta|^2V(x)-\lambda\bigr) u(x)}\mbox{d}x=0
$$
for almost every $\eta$. Consider the Schr\"odinger operator $H_\eta=-\Delta +|\eta|^2V(x)$
which, under assumption \ref{A2} is essentially self-adjoint so that, if $g\in L^2(\R^n)$,
then
$$
\int_{\R^n} g(x)\overline{\bigl(H_\eta-\lambda\bigr) u(x)}\mbox{d}x=0
$$
for every $u\in\mathcal{C}^\infty_c(\R^n)$ implies that $g=0$ a.e.
We conclude that $\fc_2[f](x,\eta)=0$ for almost every $x$ and almost every $\eta$ so that $f=0$ a.e.

On $L^2(\R^n\times\T^m)$, we replace partial Fourier transform with partial Fourier coefficients.
\end{proof}

\section{Proof of the exact observability inequalities}\label{sec2d}
\label{secfinalproof}

We are now ready to prove the main theorem.
Let us recall what we want to prove in the synthetic table, Table \ref{table.1}.

\begin{table}[h!]
	\centering
\begin{tabular}{|c|c|c|}
	\hline
	& Assumption~\ref{A2} & Assumption~\ref{A1}\\
	\hline
	\multirow{2}{7em}{$s> (\beta_1+2)/3 $} & $s>(\beta +2) /3$, exactly null-controllable &  \\
						    & for any $T>0$ & exactly null-controllable \\
	\cline{1-2}
	\multirow{2}{7em}{$s> (\beta_1+2)/4 $}					 &  & \\
						 & $s\le (\beta_1+2)/3$  & \\
						 \cline{1-1} \cline{3-3}
	\multirow{2}{7em}{$ s\le(\beta _1+2) /4 $} & not known under Assumption~\ref{A2} & not known under Assumption~\ref{A1} \\
						&  {\it see} Remark~\ref{r1h} for the standard case & {\it see} Remark~\ref{r1h} for the standard case \\
	\hline
\end{tabular}

\smallskip

\caption{Exactly null-controllability results from equidistributed sets. For the case $y\in\R^m$, we need to set $\beta_1=\beta_2$, while for the case $y\in\T^m$ we do not need this condition.}
\label{table.1}
\end{table}

As already mentionned, it is enough to prove the observability properties.
Let us start with the non-periodic case. The observability equation corresponding to
\eqref{egb} is the following:

\begin{theorem} Let $T>0$, $\beta_1=\beta_2>0$, and $\gamma>0$.
	Let $V\in L^1_{\mathrm{loc}}(\R^n)$ that satisfies Assumption~\ref{A1} (resp. Assumption~\ref{A2}) and let $\lc_V=-\Delta_x-V(x)\Delta_y$
be the corresponding Baouendi-Grushin operator on $\R^n\times\R^m$. Let $\omega\subset\R^n$
be a $\gamma$-equidistributed set. If $s>s_{A1}$ (resp. $s>s_{A2}$), then the there exists a constant $C_{\mathrm{obs}}^{\lc}(T,s,V,\omega)$
such that the inequality
$$
\norm{e^{-T\lc_V^s}f}_{L^2(\R^n\times\R^m)}
\leq C_{\mathrm{obs}}^{\lc}(T,s,V,\omega)\int_0^T\norm{e^{-t\lc_V^s}f}_{L^2(\omega\times\R^m)}\,\mathrm{d}t
$$
holds for every $f\in L^2(\R^n\times\R^m)$.
\end{theorem}

\begin{proof}
We have, using Fubini, Parseval (for $\fc_2$) and Proposition \ref{prp2.2}
\begin{eqnarray*}
\norm{e^{-T\lc_V^s}u}_{L^2(\R^n\times\R^m)}^2&=&
\norm{\fc_2[e^{-T\lc_V^s}u]}_{L^2(\R^n\times\R^m)}^2
=\int_{\R^m}\int_{\R^n}\abs{e^{-TH_{|\eta|^2V}^s}\fc_2[u](x,\eta)}^2\,\mbox{d}x\,\mbox{d}\eta\\
&\leq&\int_{\R^m}\frac{C_{\mathrm{obs}}(T,s,|\eta|^2V,\omega)}{T}
\int_0^T\int_{\omega}\abs{e^{-tH_{|\eta|^2V}^s}\fc_2[u](x,\eta)}^2\,\mbox{d}x\,\mbox{d}\eta\,\mbox{d}t\\
&\leq&\frac{\sup_{r>0}C_{\mathrm{obs}}(T,s,rV,\omega)}{T}\int_0^T\int_{\omega}
\int_{\R^m}\abs{e^{-tH_{|\eta|^2V}^s}\fc_2[u](x,\eta)}^2\,\mbox{d}x\,\mbox{d}\eta\,\mbox{d}t\\
&=&\frac{\sup_{r>0}C_{\mathrm{obs}}(T,s,rV,\omega)}{T}\int_0^T\int_{\omega}
\int_{\R^m}\abs{e^{-t\lc_V^s}u(x,y)}^2\,\mbox{d}x\,\mbox{d}y\,\mbox{d}t
\end{eqnarray*}
provided
$$
C_{\mathrm{obs}}^{\lc}(T,s,V,\omega):=\frac{1}{T}\sup_{r>0}C_{\mathrm{obs}}(T,s,rV,\omega)
<+\infty.
$$
It remains to use Proposition \ref{prp2.2} and Lemma \ref{lem:prp2.2}
to conclude that the observability inequality holds in the listed case.
\end{proof}

\begin{remark}
The proof actually also provides an estimate of the observability constants, namely, when $s>s_{A1}$ (resp. $s>s_{A2}$) and $C_{\mathrm{obs}}^{\lc}(T,s,V,\omega)\lesssim\dfrac{B_-(T,s,\omega,V)}{T}$
when $T<1$ and $C(T,s,\omega,V)\lesssim\dfrac{B_+(T,s,\omega,V)}{T}$ when $T\geq 1$ where
$B_-$ and $B_+$ are
defined in \eqref{eq:b><}-\eqref{eq:b>>};
\end{remark}

Next, we show that adding a zero-order term to $\lc_V$ allows to obtain observability 
corresponding to \eqref{egbs} from smaller sets:

\begin{proposition} Let $T>0$, $\sigma\geq0$, and $\gamma>0$.
	Let $V,\tilde V\in L^1_{\mathrm{loc}}(\R^n)$ that satisfy Assumption \ref{A2} (resp. Assumption~\ref{A1})
with same parameters $\beta_1,\beta_2$ and let $\lc_{V,\tilde V}=-\Delta_x-V(x)\Delta_y+\tilde V$
be the corresponding Baouendi-Grushin-Schr\"odinger operator on $\R^n\times\R^m$. Let $\omega\subset\R^n$
be a $(\gamma,\sigma)$-distributed set. Assume further that $s,\beta_1,\beta_2,\sigma,T$ satisfy (ii) or (iii) with in Lemma~\ref{lem:prp2.2}. 

Then the there exists a constant $C_{\mathrm{obs}}^{\lc}(T,s,V,\tilde V,\omega)$
such that for every $f\in L^2(\R^n\times\R^m)$
$$
\norm{e^{-T\lc_{V,\tilde V}^s}f}_{L^2(\R^n\times\R^m)}
\leq C_{\mathrm{obs}}^{\lc}(T,s,V,\tilde V,\omega)
\int_0^T\norm{e^{-t\lc_{V,\tilde V}^s}f}_{L^2(\omega\times\R^m)}\,\mathrm{d}t.
$$
\end{proposition}

\begin{proof}
We have, using Fubini, Parseval (for $\fc_2$) and Proposition \ref{prp2.2}
\begin{eqnarray*}
\norm{e^{-T\lc_{V,V_0}^s}u}_{L^2(\R^n\times\R^m)}^2&=&
\norm{\fc_2[e^{-T\lc_{V,V_0}^s}u]}_{L^2(\R^n\times\R^m)}^2
=\int_{\R^m}\int_{\R^n}\abs{e^{-TH_{|\eta|^2V+V_0}^s}\fc_2[u](x,\eta)}^2\,\mbox{d}x\,\mbox{d}\eta\\
&\leq&\int_{\R^m}\frac{C_{\mathrm{obs}}(T,s,(1+|\eta|^2)V,\omega)}{T}
\int_0^T\int_{\omega}\abs{e^{-tH_{|\eta|^2V+V_0}^s}\fc_2[u](x,\eta)}^2\,\mbox{d}x\,\mbox{d}\eta\,\mbox{d}t.
\end{eqnarray*}
Here we use Remark \ref{rem:vtildev} to absorb $\tilde V$ into $V$. We thus get
$$
\norm{e^{-T\lc_{V,V_0}^s}u}_{L^2(\R^n\times\R^m)}^2
\leq
\frac{\sup_{r\geq 1}C_{\mathrm{obs}}(T,s,rV,\omega)}{T}\int_0^T\int_{\omega}
\int_{\R^m}\abs{e^{-t\lc_{V,V_0}^s}u(x,y)}^2\,\mbox{d}x\,\mbox{d}y\,\mbox{d}t
$$
provided
$$
C_{\mathrm{obs}}^{\lc}(T,s,V,\tilde V,\omega):=\frac{1}{T}\sup_{r>1}C_{\mathrm{obs}}(T,s,rV,\omega)<+\infty.
$$

It remains to use Lemmas \ref{lem:prp2.2} and \ref{lem:cond} to conclude.
\end{proof}

We finally treat the semi-periodic case:

\begin{theorem}
	Let $T>0, \beta_2\ge \beta_1>0$, and $\gamma>0$. Let $V\in L^{1}_{\mathrm{loc}}(\R^{n})$ that satisfies Assumption~\ref{A1} (resp. Assumption~\ref{A2}) and let $\mathcal{L}_{V}=-\Delta_x-V(x)\Delta_y$ be the corresponding Baouendi-Grushin operator on $\R^n\times\T^m$. Let $\omega \subset \R^{n}$ be a $\gamma$-equidistributed set. If $s>s_{A1}$ (resp. $s>s_{A2}$), then there exists a constant $C_{\mathrm{obs}}^{\mathcal{L}}(T,s,V,\omega)$ such that the inequality
$$
\norm{e^{-T\lc_V^s}f}_{L^2(\R^n\times\T^m)}
\leq C_{\mathrm{obs}}^{\lc}(T,s,V,\omega)\int_0^T\norm{e^{-t\lc_V^s}f}_{L^2(\omega\times\T^m)}\,\mathrm{d}t
$$
holds for every $f \in L^2(\R^{n}\times \T^{m})$.
\end{theorem}

\begin{proof}
	We have, using Fubini and Parseval (for $\mathcal{F}_2^p$)
	\begin{equation*}
			\|e^{-T \mathcal{L}_{V}^{s}}u\|^2_{L^2(\R^{n}\times \T^{m})}
		= \|\mathcal{F}_2^p[ e^{-T \mathcal{L}^{s}_{V}}u]\|^2_{L^2(\R^{n}\times \T^{m})}
		=\sum_{k \in \Z^{m}}\int_{\R^{n}}\left| e^{-TH^{s}_{|k|^2 V}}\mathcal{F}_2^p[u](x,k) \right| ^2 \d x.
	\end{equation*}

For $k=0$, we use that a $\gamma$-equidistributed set is also $\gamma$-thick. Then, as 
$H_{0V}=\Delta$ we can appeal to Theorem~\ref{th:WWZZEV} to bound
$$
\int_{\R^{n}}\left| e^{-TH^{s}_{0V}}\mathcal{F}_2^p[u](x,k) \right| ^2 \d x
\leq C\int_0^T\int_{\omega}\left| e^{-TH^{s}_{0V}}\mathcal{F}_2^p[u](x,k) \right| ^2 \d x\d t.
$$
	
For  $k\neq 0$, we have $|k|\ge 1$. Then we apply Lemma~\ref{lem:prp2.2}, Lemma~\ref{lem:cond} and Proposition~\ref{prp2.2} to obtain
$$
\sum_{k \in \Z^{m}\setminus\{0\}}\int_{\R^{n}}\left| e^{-TH^{s}_{|k|^2 V}}\mathcal{F}_2^p[u](x,k) \right| ^2 \d x
\leq C\sum_{k \in \Z^{m}\setminus\{0\}}\int_0^T\int_{\omega}\left| e^{-TH^{s}_{|k|^2 V}}\mathcal{F}_2^p[u](x,k) \right| ^2 
\d x\d t.
$$
Grouping both estimates and applying Parseval's relation gives us the desired observability inequality. 
\end{proof}

Note that the $k=0$ case in the above proof requires $\omega$ to be $\gamma$-thick so that, for this
proof to work for $(\gamma,\sigma)$-distributed sets, one needs $\sigma=0$.

\section{Data availability}
No data has been generated or analysed during this study.

\section{Funding and/or Conflicts of interests/Competing interests}

The authors have no relevant financial or non-financial interests to disclose.

Ce travail a bénéficié d'une aide de l'\'Etat attribu\'e \`a l'Universit\'e de Bordeaux en
tant qu'Initiative d'excellence, au titre du plan France 2030.

\bibliographystyle{alpha}

\end{document}